\documentclass{amsart}
\usepackage{amssymb}
\usepackage{tikz-cd}
\usepackage{verbatim}
\usepackage{hyperref}
\DeclareMathOperator{\In}{In}
\DeclareMathOperator{\Bun}{Bun}
\DeclareMathOperator{\Aut}{Aut}
\DeclareMathOperator{\uAut}{\underline{Aut}}
\newcommand{\Isom}{\underline{\operatorname{Isom}}}
\DeclareMathOperator{\Dist}{HaarDist}
\DeclareMathOperator{\HDist}{HaarDist}
\DeclareMathOperator{\GL}{GL}
\DeclareMathOperator{\id}{id}
\DeclareMathOperator{\Spd}{Spd}
\DeclareMathOperator{\Spa}{Spa}
\DeclareMathOperator{\Spec}{Spec}

\DeclareMathOperator{\Hecke}{Hecke}
\newcommand{\LHecke}{\Hecke^{\mathrm{loc}}}
\DeclareMathOperator{\Haar}{Haar}
\DeclareMathOperator{\Rel}{Rel}
\DeclareMathOperator{\tr}{tr}
\DeclareMathOperator{\rk}{rk}
\DeclareMathOperator{\inv}{inv}
\DeclareMathOperator{\cc}{cc}
\DeclareMathOperator{\rank}{rank}
\DeclareMathOperator{\Fix}{Fix}
\DeclareMathOperator{\Groth}{Groth}
\DeclareMathOperator{\Hom}{Hom}
\DeclareMathOperator{\Div}{Div}
\DeclareMathOperator{\Gal}{Gal}

\DeclareMathOperator{\Frob}{Frob}
\DeclareMathOperator{\Lie}{Lie}
\DeclareMathOperator{\Ad}{Ad}

\DeclareMathOperator{\coev}{coev}
\DeclareMathOperator{\ev}{ev}
\DeclareMathOperator{\Coh}{Coh}
\DeclareMathOperator{\Rep}{Rep}
\newcommand{\isom}{\xrightarrow{\sim}}

\newcommand{\BC}{\mathcal{BC}}
\DeclareMathOperator{\RHom}{RHom}
\newcommand{\uRHom}{\underline{\RHom}}
\newcommand{\EE}{\mathcal{E}}
\newcommand{\Qp}{\mathbb{Q}_p}

\newcommand{\Qlb}{\overline{\mathbb{Q}}_{\ell}}
\newcommand{\et}{\mathrm{\acute{e}t}}

\newcommand{\Fpb}{\overline{\mathbb{F}}_p}
\newcommand{\OO}{\mathcal{O}}
\newcommand{\ZZ}{\mathbb{Z}}
\newcommand{\Zl}{\ZZ_{\ell}}

\newcommand{\rs}{\mathrm{rs}}
\newcommand{\sr}{\mathrm{sr}}
\newcommand{\an}{\mathrm{an}}

\newcommand{\loc}{\mathrm{loc}}
\newcommand{\Sht}{\mathrm{Sht}}
\newcommand{\lis}{\mathrm{lis}}
\newcommand{\trdist}{\mathrm{tr.\ dist}}

\newcommand{\Set}{\mathrm{Set}}
\newcommand{\alg}{\mathrm{alg}}
\newcommand{\ol}{\overline}
\newcommand{\Gbt}{\widetilde{G}_b}
\newcommand{\Gbu}{\underline{G_b(F)}}
\newcommand{\Gbpu}{\underline{G_{b'}(F)}}
\newcommand{\Perf}{\mathrm{Perf}}
\newcommand{\Perfb}{\Perf_{\Fpb}}
\newcommand{\Fp}{\mathbb{F}_p}

\newcommand{\Gm}{\mathbb{G}_m}

\newcommand{\pr}{p}

\newcommand{\reg}{\mathrm{reg}}
\newcommand{\VV}{\mathcal{V}}
\newcommand{\abs}[1]{\left\lvert #1 \right|}

\numberwithin{equation}{subsection}
\newtheorem{thm}[equation]{Theorem}
\newtheorem{lem}[equation]{Lemma}
\newtheorem{prop}[equation]{Proposition}
\newtheorem{cor}[equation]{Corollary}

\theoremstyle{definition}
\newtheorem{defn}[equation]{Definition}
\newtheorem{defprop}[equation]{Definition/Proposition}

\theoremstyle{remark}

\newtheorem{ex}[equation]{Example}
\title{On the strongly regular locus of the inertia stack of $\Bun_G$}
\author{Daniel R. Gulotta}
\address{Department of Mathematics, University of Utah, Salt Lake City, UT 84112, USA}
\email{dgulotta@alum.mit.edu}
\begin{document}
\begin{abstract}
Let $G$ be a connected reductive group over a finite extension of $\mathbb{Q}_p$.  We show that
for each $b \in B(G)$, the strongly regular locus of the inertia stack of $\operatorname{Bun}_G^b$ is open
in the inertia stack of $\operatorname{Bun}_G$.  As a consequence, we extend the computation of Hansen--Kaletha--Weinstein
of trace distributions of the cohomology of local shtuka spaces $\mathrm{Sht}_{G,b,\mu}$ to non-basic $b$.
If $b$ is closed in $B(G,\mu)$, or $b$ is basic and has only one specialization in $B(G,\mu)$,
then we compute the trace distribution of the entire strongly regular locus.
In the process, we prove some results on the behavior of characteristic
classes under cohomologically smooth pullback.
\end{abstract}
\maketitle
\tableofcontents
\section{Introduction}
Let $F$ be a finite extension of $\Qp$, and let $G$ be a connected
reductive group over $F$.
Scholze \cite{berkeley} has defined a tower of moduli spaces of mixed characteristic local shtukas
\[ \Sht_{G,b,\mu} = \varprojlim_K \Sht_{G,b,\mu,K} \,. \]
Here, $\mu$ is a conjugacy class of cocharacters $\mu \colon \mathbb{G}_m \to G$ defined over $\overline{F}$,
$b$ is an element of the Kottwitz set $B(G,\mu)$,
and $K$ ranges over open compact subgroups of $G(F)$.
Let $E$ be the field of definition of the conjugacy class of $\mu$. Then $\Sht_{G,b,\mu,K}$ is a locally
spatial diamond over $\operatorname{Spd} \breve{E}$, where $\breve{E}$ is the completion of the maximal
unramified extension of $E$.

The tower $\Sht_{G,b,\mu,K}$ has an action of $G(F) \times G_b(F)$, where $G_b(F)$ is the automorphism
group of the isocrystal associated with $b$.

The cohomology of $\Sht_{G,b,\mu}$ is of interest in the local Langlands program.
Let $\ell$ be a prime different from $p$.  The geometric Satake equivalence attaches to each $\mu$
an object $\mathcal{S}_{\mu}$ in the equivariant derived category of \'etale $\Zl$-sheaves on $\Sht_{G,b,\mu,K}$.
Let $C$ be the completion of the algebraic closure of $F$.
For any smooth representation $\rho$ of $G_b(F)$ with coefficients in $\Qlb$, define
\[ R \Gamma(G,b,\mu)[\rho] = \varinjlim_K R\Hom_{G_b(F)}(R\Gamma_c(\Sht_{G,b,\mu,K,C},\mathcal{S}_{\mu}),\rho) \,. \]
Fargues and Scholze \cite[Corollary I.7.3]{fargues-scholze} have shown that if
$\rho$ is a finite length admissible representation of $G_b(F)$, then
$R \Gamma(G,b,\mu)[\rho]$ is represented by a complex of $G(F) \times W_E$-representations that are
finite length and admissible as representations of $G(F)$.

One way to study $R \Gamma(G,b,\mu)[\rho]$ explicitly is via trace distributions.
Recall that for any admissible smooth representation $\pi$ of $G(F)$ and any
compactly supported smooth function $f \colon G(F) \to \Qlb$, the endomorphism of $\pi$ given by
$\int_{G(F)} f(g) \pi(g) \, dg$ has finite-dimensional image,
so it makes sense to define its trace
$\tr(f|\pi)$.
If $\pi$ also has finite length, then it has a Harish-Chandra character
$\Theta_{\pi}$,
which
is a smooth function from the regular semisimple locus $G(F)_{\rs}$ to $\Qlb$
such that
\[ \tr(f|\pi) = \int_{G(F)_{\rs}} f(g) \Theta_{\pi}(g) \, dg \text{ for all $f$.} \]
Similarly, a finite length admissible smooth representation $\rho$ of
$G_b(F)$ has a Harish-Chandra character $\Theta_{\rho}$.
If an object in the derived category $D(G(F),\Qlb)$
has finite length admissible cohomology groups, then its
Harish-Chandra character is defined
to be the alternating sum of the Harish-Chandra
characters of the cohomology groups.

For any finite length admissible smooth representation $\rho$ of $G_b(F)$, 
Hansen--Kaletha--Weinstein compare the elliptic parts of the Harish-Chandra characters of $\rho$ and $R\Gamma(G,b,\mu)[\rho]$.
\begin{thm}[{\cite[Theorem 6.5.2]{hkw}}] \label{hkw-649}
Assume $b$ is basic.  Let $\rho$ be a finite length admissible smooth representation of $G_b(F)$, and let
$\pi = R\Gamma(G,b,\mu)[\rho]$.

Let $g \in G(F)$ be an elliptic element.  Then
\[ \Theta_{\pi}(g) = (-1)^{\left<\mu,2\rho_G\right>} \sum_{(g,g',\lambda) \in \Rel_b} \dim r_{\mu}[\lambda] \Theta_{\rho}(g') \,. \]
\end{thm}
Here, $\rho_G$ is the sum of the positive roots of $G$,
$\Rel_b$ is a set parametrizing pairs of stably conjugate elements
of $G(F)$ and $G_b(F)$, plus some additional data,
and $r_{\mu}[\lambda]$ is the $\lambda$-isotypic subspace of
the representation of the Langlands dual group
$\widehat{G}$ with highest
weight $\mu$.
If $\rho$ belongs to a supercuspidal $L$-packet, then this identity confirms a prediction of the Kottwitz and refined local Langlands conjectures.

Recall that an element of a reductive group is \emph{strongly regular} if its centralizer is a torus.
It would be useful to generalize Theorem \ref{hkw-649} to strongly regular $g$, since the strongly regular
locus is dense in $G(F)$, while the elliptic locus may not be dense.
We make some progress toward this goal.
\begin{thm}[Theorem \ref{trace}] \label{trace intro}
Let $\rho$ be a finite length admissible representation of $G_b(F)$, and let $\pi = R\Gamma(G,b,\mu)[\rho]$.
Let $g \in G(F)$ be a strongly regular element.
Suppose that for all specializations $b'$ of $b$ in $B(G,\mu)$, $g$ is not stably conjugate to any
element of $G_{b'}(F)$.  Then
\[ \Theta_{\pi}(g) = (-1)^{\left<\mu,2\rho_G\right>} \sum_{(g,g',\lambda) \in \Rel_b} \dim r_{\mu}[\lambda] \Theta_{\rho}(g') \prod_{\alpha \in \Phi^+} |1-\alpha(\lambda)|^{-1} \,. \]
Here, $\Phi^+$ is the set of roots of $G$ that are positive with respect to the parabolic subgroup of $G$ associated with $b$,
and the norm is defined so that $\abs{p}=p^{-[F:\Qp]}$.
\end{thm}
Our definition of $\Rel_b$ is slightly different from the one in \cite{hkw}, but the elliptic locus of our $\Rel_b$ is the same as the elliptic locus of theirs.
\begin{ex}
Let $M$ be a Levi subgroup of $G$, let $\mu$ be a cocharacter of $G$ whose centralizer is $M$,
and let $b=\mu(p)$.  Then $G_b = M$.  After accounting for normalization,
Theorem \ref{trace intro} recovers van Dijk's formula for the Harish-Chandra
character of a parabolically induced representation \cite[Theorem 3]{vandijk}.
\end{ex}

A key step in the proof of Theorem \ref{trace intro} is the following result:
\begin{thm}[Theorem \ref{sr open}] \label{sr open intro}
For each $b \in B(G)$, the strongly regular locus of the inertia stack
of $\Bun_G^b$ is open in the inertia stack of $\Bun_G$.
\end{thm}
Here, $\Bun_G$ is the stack classifying $G$-bundles on the Fargues--Fontaine curve.
This theorem also holds when $F = \mathbb{F}_q((T))$.

The reason that Theorem \ref{trace intro} does not cover the entire strongly regular locus is the following.
An admissible representation $\rho$ of $G_b(F)$ can be considered as a constructible sheaf on $\Bun_G^b$,
and the trace distribution of $\rho$ can be regarded as the characteristic class of this sheaf.
One can also consider the characteristic class of the sheaf $i_{b*} \rho$ on $\Bun_G$.
Theorem \ref{sr open intro} implies that this characteristic class determines
a trace distribution on the strongly
regular locus of $G_{b'}(F)$ for each $b' \in B(G)$.
The $b$-component is the usual trace distribution of $\rho$.
If $b'$ is a specialization of $b$,
then the $b'$-component may be nonzero.  These components show up
in the computation of $\Theta_{\pi}$.

The problem of computing these $b'$-components seems interesting in its own right.
We compute some of these components in Examples \ref{gl2 example}-\ref{depth one example},
but we do not have a general formula.
In the case where $\Sht_{G,b,\mu}$ is the Drinfeld upper half plane,
Example \ref{gl2 example hecke} uses the result of Example \ref{gl2 example}
to compute $\Theta_{\pi}$ over the entire strongly regular locus.

In order to determine the factor of
$\prod_{\alpha \in \Phi^+} \abs{1-\alpha(\lambda)}^{-1}$ appearing in
Theorem \ref{trace intro}, and to work out Examples
\ref{gl2 example}-\ref{depth one example}, we establish a method for computing
the behavior of characteristic classes under cohomologically smooth pullback.
This method is developed in section \ref{cc smooth}.

\subsection{Outline of the paper}
In section \ref{bundles general}, we prove some general results about endomorphisms of $G$-bundles. 
In Section \ref{sr}, we use these results to prove Theorem \ref{sr open intro}.
In Section \ref{char class section}, we prove some results on the behavior of characteristic classes under
pullback, and use these to compute characteristic classes on $\Bun_G$.
In Section \ref{trace section}, we explain how to modify the argument of
\cite{hkw} to prove Theorem \ref{trace intro}.

\subsection*{Acknowledgments}
I would like to thank David Hansen and Jared Weinstein for helpful discussions.
I would also like to thank the anonymous referee for some corrections.
This work was supported in part by the Simons Foundation (Grant Number
814268 via the Mathematical Sciences Research Institute, MSRI), the
Deutsche Forschungsgemeinschaft (DFG, German Research Foundation) under Germany's
Excellence Strategy -- EXC-2047/1 -- 390685813,
National Science Foundation Grant No.~DMS-1840190,
and a C.~E.~Burgess Instructorship at the University of Utah.
I would also like to thank Boston University for its hospitality.

\section{Automorphisms of $G$-bundles in general} \label{bundles general}
Let $F$ be a finite extension of $\Qp$ or $\Fp((t))$, and let $G$ be a connected
reductive group over $F$.
We will sometimes abuse notation and write $G$ for the diamond
$(G^{\an})^{\lozenge}$, and do similarly with other schemes.

\subsection{Centralizers and bundles}
Since we are interested in $G$-bundles on the Fargues--Fontaine curve,
it is natural to work in the category of sousperfectoid spaces
defined in \cite{hansen-kedlaya-sheafiness} (see also \cite[\S 6.3]{berkeley}).
There are several equivalent ways of defining $G$-bundles
on a sousperfectoid space.
\begin{defprop}[{\cite[Theorem 19.5.2]{berkeley},
\cite[Definition/Proposition III.1.1]{fargues-scholze}}]
\label{bundle def}
Let $Y$ be a sousperfectoid space.
The following categories are naturally equivalent:
\begin{enumerate}
\item The category of adic spaces $T \to Y$ with a $G$-action
such that \'etale locally on $Y$, there is a $G$-equivariant
isomorphism $T \cong G \times Y$.
\item The category of \'etale sheaves $\EE$ on $Y$ with
a $G$-action such that \'etale locally, $\EE \cong G$. \label{sheaf def}
\item The category of exact $\otimes$-functors from the category $\Rep_F G$
of algebraic $F$-vector space representations of $G$ to the category $\Bun(Y)$ of vector bundles
on $Y$. \label{tannakian def}
\end{enumerate}
We define a \emph{$G$-bundle} on $Y$ to be an object in the category (3).
We will sometimes also view it as an object in one of the
first two categories.
\end{defprop}

\begin{defn}[{\cite[Definition 5.3.2]{berkeley}}]
Let $Y$ be a sousperfectoid space.  A \emph{Cartier divisor}
on $Y$ is is an ideal sheaf $\mathcal{I} \subset \OO_Y$ that is locally
free of rank $1$.
\end{defn}

\begin{defn} \label{modification def}
Let $Y$ be a sousperfectoid space, let $\mathcal{I}$
be a Cartier divisor on $Y$, and let $U$ be the complement
of the support of $\OO_Y/\mathcal{I}$.

Let $\VV$, $\VV'$ be vector bundles on $Y$.
A \emph{modification} $\EE \dashrightarrow\EE'$
along $\mathcal{I}$ is an isomorphism of vector bundles
$\VV|_U \isom \VV'|_U$
that extends to a map of sheaves of $\OO_Y$-modules
$\VV \to \varinjlim_n \VV' \otimes \mathcal{I}^{\otimes-n}$.

Let $\EE$, $\EE'$ be $G$-bundles on $Y$.
A \emph{modification} $\EE \dashrightarrow \EE'$
along $\mathcal{I}$ is an isomorphism of $G$-bundles
$\EE|_U \isom \EE'|_U$ such that for any object of $\Rep_F G$,
the associated map of vector bundles is a modification along $\mathcal{I}$.
\end{defn}

\begin{prop} \label{bundle centralizer}
Let $Y$ be a sousperfectoid space over $F$, and let $\EE$, $\EE'$ be
$G$-bundles on $Y$.  Let $g \in \Aut \EE$, $g' \in \Aut \EE'$.
Let $\widetilde{\EE} \subset \underline{\Aut} \EE$ denote the centralizer of $g$, and let
$\widetilde{\EE}'$ denote the equalizer of the maps $\Isom(\EE,\EE') \to \Isom(\EE,\EE')$ given by $x \mapsto gx$ and $x \mapsto xg'$.

Suppose that
\'etale locally on $Y$, $g$ and $g'$ are conjugate.
Also  suppose that we are given a connected reductive group $H$ over $F$
and an isomorphism $h \colon H \times_F Y \isom \widetilde{\EE}$ of groups over $Y$.

Then:
\begin{enumerate}
\item $\widetilde{\EE}'$ is an $H$-bundle.
\item There is an isomorphism of $G$-bundles $\EE' \cong (\widetilde{\EE}' \times_Y \EE)/H$.
\end{enumerate}

Further suppose that there exists a finite extension $K$ of $F$ and a homomorphism
$h' \colon H_K \to G_K$ such that the induced map $H_K \times_F Y \to G_K \times_F Y$ is \'etale locally
conjugate to $h_K$.
Then:
\begin{enumerate}
\setcounter{enumi}{2}
\item For any Cartier divisor $\mathcal{I}$ on $Y$,
the isomorphism (2) induces a bijection between modifications of $G$-bundles
$\EE \dashrightarrow \EE'$ along $\mathcal{I}$ intertwining $g$ and $g'$ and modifications
of $H$-bundles $\widetilde{\EE} \dashrightarrow \widetilde{\EE}'$ along $\mathcal{I}$.
\end{enumerate}
\end{prop}
\begin{proof}
Items (1) and (2) are clear from Definition/Proposition \ref{bundle def}(\ref{sheaf def}).

Now we will prove item (3).  Let $U$ be the complement of the support of $\OO_Y/\mathcal{I}$.
It is clear that (2) induces a bijection between
isomorphisms $\EE|_U \cong \EE'|_U$ intertwining $g$ and $g'$
and isomorphisms $\widetilde{\EE}|_U \cong \widetilde{\EE}'|_U$.
It remains to check that this bijection preserves the property of being meromorphic along $\mathcal{I}$.

By \'etale descent for vector bundles \cite[Theorem 8.2.22(d)]{kl-relative},
we can reduce to the case where $\EE$, $\EE'$ are trivial and $K=F$.
Let $U$ be the complement of the support of $\OO_Y/\mathcal{I}$.
Modifications $\EE \dashrightarrow \EE'$ along $\mathcal{I}$ are then
identified with maps $\phi \colon U \to G$ such that for every
finite-dimensional algebraic representation $r \colon G \to GL(V)$
(equivalently, for some faithful $r$), $r \circ \phi$ induces
an automorphism of $V \otimes_F \OO_U$ that is meromorphic along $\mathcal{I}$.
If $r$ is faithful, then $r \circ h'$ is also faithful.
Conjugating an automorphism of $V \otimes_F \OO_U$ by an automorphism
of $V \otimes_F \OO_Y$ will not change whether the former is meromorphic along $\mathcal{I}$.
\end{proof}

\subsection{The coarse quotient $G//G$, and the strongly regular locus}

Consider the action of $G$ on itself by conjugation, and let
$G//G = \Spec \left(\Gamma(G,\OO_G)^G \right)$ denote the coarse quotient.

Any morphism $H \to G$ of connected reductive groups over $F$ induces a morphism
$H//H \to G//G$.

Let $Y$ be a scheme or sousperfectoid space over $F$, and let $\EE$ be a $G$-bundle over $Y$.
Suppose we are given an \'etale covering $Z \to Y$ along with a trivialization
$\tau \colon \EE \times_Y Z \xrightarrow{\sim} G \times_F Z$.
The action $\uAut \EE \times_Y (G \times_F Z) \to (G \times_F Z)$
is of the form $(h,g,z) \mapsto (g \gamma_{\tau}(h,z),z)$ for some
$\gamma_{\tau} \colon \uAut \EE \times_Y Z \to G$.
The map $\gamma_{\tau}$ depends $\tau$ only up to conjugation.
Moreover, the composite
\[
\uAut \EE \times_Y Z \xrightarrow{\gamma_{\tau}} G \to G//G
\]
descends to a map $\chi_{\EE} \colon \uAut \EE \to G//G$.
\begin{ex} \label{isom example}
If $G'$ is an inner form of $G$, then $\Isom(G,G')$ is a $G$-bundle over $F$, with
automorphism group $G'$.  There is an induced map $G' \to G//G$.
\end{ex}

\begin{defn}
Let $K$ be an extension of $F$.  An element of $G(K)$ is \emph{strongly regular} if its centralizer
is a torus.
\end{defn}

\begin{lem} \label{sr pullback}
The locus of strongly regular points of $G$ is an open subscheme $G_{\sr}$ of $G$.
This open subscheme is the pullback of an open subscheme $(G//G)_{\sr}$ of the coarse quotient $G//G$.
\end{lem}
\begin{proof}
Since the property of being strongly regular is insensitive to field extensions, there
is no harm in assuming that $G$ is split.  Let $T$ be a split maximal torus of $G$,
and let $W$ be the Weyl group of $G$.

Let $g$ be a geometric point of $G$.  If $g$ is strongly regular, then it is
regular semisimple, i.e.~the multiplicity of the eigenvalue $1$ in the adjoint action
of $g$ on $\mathfrak{g}$ is as small as possible.  It is clear that the regular
semisimple locus is the pullback of an open subscheme of $G//G$.

If $g$ is regular semisimple, then it is
conjugate to a geometric point of $T$.
Using the Bruhat decomposition, we see that a regular semisimple geometric point of $T$ is
strongly regular iff its
stabilizer under the action of the Weyl group $W$ is trivial.
The locus of points of $T$ with trivial stabilizer is a
$W$-stable open subspace of
$T$.  By \cite[Proposition V.1.1(ii)]{SGA1}, this
locus is the pullback of an open subspace of $\Spec \left(\Gamma(T,\OO_T)^W\right)$.

By Chevalley restriction, the map $\Gamma(G,\OO_G)^G \to \Gamma(T,\OO_T)^W$ is an isomorphism.
So the strongly regular locus of $G$ is the pullback of an open subscheme
of $G//G$.
\end{proof}

\begin{defn}
Let $X$ be a scheme, sousperfectoid space, or v-stack over $G//G$.
The \emph{strongly regular locus} of $X$, denoted $X_{\sr}$, is the pullback $X \times_{G//G} (G//G)_{\sr}$.
\end{defn}

Two strongly regular elements of $G$ are stably conjugate if and only
if they have the same image in $G//G$.  The following lemma generalizes this fact.
\begin{lem} \label{etale local section}
The map $G \times_F G_{\sr} \to G_{\sr} \times_{G//G} G_{\sr}$ given by
$(g,h) \mapsto (h,ghg^{-1})$ admits an \'etale local section.
\end{lem}
\begin{proof}
We claim the map is smooth and surjective.  Indeed, it is straightforward
to check smoothness using Chevalley restriction, and \cite[Corollary 6.6]{steinberg-regular} implies surjectivity.  Then an \'etale local section exists by
\cite[Corollaire 17.16.3(ii)]{EGA44}.

Alternatively, one can use the Bruhat decomposition to give an explicit
\'etale local section.  Choose a finite extension $K/F$ such that
$G_K$ is split, and let $B$ and $\overline{B}$ be opposite Borel subgroups of
$G_K$.  Let $T = B \cap \ol{B}$, and let $N$ and $\ol{N}$ be the unipotent
radicals of $B$ and $\ol{B}$, respectively.  Let $W$ be the Weyl group of
$G_K$.  Recall that
$G_K = \bigsqcup_{w \in W} Bw\ol{B}$, and that $Bw\ol{B}\subseteq B\ol{B}w = N \ol{N} w T$.
Therefore, $G/T = \cup_{w \in W} N \ol{N} w T/T$.
Choose a representative in the normalizer of $T$ for each element of $W$.
Then there is an \'etale covering
\[ T_{\sr} \times_K N \times_K \ol{N} \times_K N \times_K \ol{N} \times W  \to G_{\sr} \times_{G//G} G_{\sr} \]
given by
\[ (t,n_1,\bar{n}_1,n_2,\bar{n}_2,w) \mapsto
(n_1 \bar{n}_1 t \bar{n}_1^{-1} n_1^{-1},
n_2 \bar{n}_2 wtw^{-1} \bar{n}_2^{-1} n_2^{-1})  \,, \]
and the map to $G \times_F G_{\sr}$ given by
\[ (t,n_1,\bar{n}_1,n_2,\bar{n}_2,w) \mapsto
(n_2 \bar{n}_2 w \bar{n}_1^{-1} n_1^{-1},
n_1 \bar{n}_1 t \bar{n}_1^{-1} n_1^{-1}) \]
is the desired \'etale local section.
\end{proof}

\begin{cor} \label{sr etale locally conj}
Let $Y$ be a sousperfectoid space over $F$, let $\EE$, $\EE'$ be
$G$-bundles over $Y$, and let $g \in \Aut \EE$, $g' \in \Aut \EE'$
be strongly regular automorphisms.
Then $g$ and $g'$ are \'etale locally conjugate iff the diagram
\[ \begin{tikzcd}
Y \arrow[r,"g"] \arrow[d,"g'"] & \uAut \EE \arrow[d,"\chi_{\EE}"] \\
\uAut \EE' \arrow[r,"\chi_{\EE'}"] & G//G
\end{tikzcd} \]
commutes.
\end{cor}

\section{Automorphisms of $G$-bundles on the Fargues--Fontaine curve} \label{sr}

\subsection{$G$-bundles on the Fargues--Fontaine curve} \label{ff}
As in the previous section, $F$ denotes a finite extension of $\Qp$ or $\Fp((t))$,
and $G$ denotes a connected
reductive group over $F$.  Let $\breve{F}$ denote the completion of the maximal unramified extension
of $F$, and let $\sigma \colon \breve{F} \to \breve{F}$ denote the Frobenius map.

Given $b \in G(\breve{F})$, let $G_b$ denote the algebraic group over $F$ whose functor of points is given by
\[ G_b(R) = \left\{ g \in G(R \otimes_F \breve{F}) \middle| b \sigma(g) = g b \right\} \,. \]

The group $G_b$ is an inner form of a Levi subgroup of the quasisplit inner form of $G$.

The Kottwitz set $B(G)$ is defined to be the set of $\sigma$-conjugacy classes of $G(\breve{F})$.
In other words, two elements $b,b' \in G(\breve{F})$ are considered to be equivalent if there exists $\gamma \in G(\breve{F})$
such that $b' = \sigma(\gamma) b \gamma^{-1}$.  The isomorphism class of the group $G_b$ depends only on
the image of $b$ in $B(G)$.

Let $\Perfb$ denote the category of perfectoid spaces over $\Fpb$.
For any affinoid perfectoid $S = \Spa(R,R^+)$ in $\Perfb$ with pseudouniformizer $\varpi$, define
\[ Y_S = (\Spa W_{\OO_F}(R^+)) \setminus \{ p[\varpi] = 0 \} \]
\[ X_S = Y_S / \Frob^{\ZZ} \,. \]
The space $X_S$ is the adic Fargues--Fontaine curve corresponding to the pair $(F,S)$.
This definition can be glued, so it makes sense to define $Y_S$ and $X_S$ for any
$S$ in $\Perfb$.

For any $b \in B(G)$ and any $S \in \Perfb$, there is a corresponding vector bundle $\EE_{b,S}$ on $X_S$.

\begin{defn}[{\cite[Definition III.0.1, Theorem III.0.2]{fargues-scholze}}]
Let $\Bun_G$ be the v-stack sending $S \in \Perfb$ to the groupoid of
$G$-bundles on $X_S$.

For any $b \in \Bun_G$, let $\Bun_G^b$ be the substack classifying
$G$-bundles that are isomorphic to $\EE_{b,x}$ at every geometric point $x$.
\end{defn}

Let $\Gbt$ denote the functor $\Perfb \to \Set$ defined by $\Gbt(S) = \Aut \EE_{b,S}$.
It is representable by a locally spatial diamond.
There are maps $\Gbu \to \Gbt \to \Gbu$, whose composition
is the identity.

\subsection{The strongly regular locus of $\In(\Bun_G)$} \label{sr open section}

\begin{lem} \label{affine locally constant}
Let $S$ be a perfectoid space over $\Fpb$.
For any map $X_S \to (G//G)^{\lozenge}$, there is a unique continuous map
$\abs{S} \to (G//G)(F)$ that makes the following diagram commute.
\[ \begin{tikzcd}
X_S \arrow[r] \arrow[d] & (G//G)^{\lozenge} \\
\underline{\abs{S}} \arrow[r] & \underline{(G//G)(F)} \arrow[u]
\end{tikzcd} \]
\end{lem}
\begin{proof}
By \cite[Proposition II.2.5(ii)]{fargues-scholze},
maps $X_S \to (\mathbb{A}^1_F)^{\lozenge}$ over $\Spd F$
are in bijection with maps $|S| \to F$.
Since $G//G$ is an affine variety,
the lemma follows.
\end{proof}

\begin{defn} \label{inertia stack def}
For any map of v-stacks $X \to S$, the \emph{inertia stack} $\In_S(X)$
is the fiber product $X \times_{X \times_S X} X$.

In the case $S=\Spd \Fpb$, we will shorten the notation to $\In(X)$.
\end{defn}
The stack $\In(\Bun_G)$ classifies $G$-bundles
equipped with an automorphism.
\begin{cor}
There is a map $\chi \colon \In(\Bun_G) \to \underline{(G//G)(F)}$ such that
for any $S \in \Perfb$, and any morphism $S \to \In(\Bun_G)$ inducing a
$G$-bundle $\EE$ over $X_S$ equipped with an automorphism $g \in \Aut \EE$,
the diagram
\[ \begin{tikzcd}
\uAut \EE \arrow[rd,"\chi_{\EE}"] \\
X_S \arrow[r] \arrow[u,"g"] \arrow[d] & G//G \\
{\underline{\abs{S}}} \arrow[r] \arrow[d] & \underline{(G//G)(F)} \arrow[u] \\
{\underline{|\In(\Bun_G)|}} \arrow[ur,"\abs{\chi}"]
\end{tikzcd} \]
commutes.
\end{cor}

Define $\In(\Bun_G)_{\sr} = \chi^{-1}(\underline{(G//G)(F)}_{\sr})$.

\begin{thm} \label{sr open}
For each $b \in B(G)$, $\In(\Bun_G^b)_{\sr}$ is open in $\In(\Bun_G)$.  In particular,
$\In(\Bun_G)_{\sr} = \bigsqcup_{b \in B(G)} \In(\Bun_G^b)_{\sr}$.
\end{thm}

The strategy of the proof will be to show that if $\EE$ is a $G$-bundle on $X_S$
admitting a strongly regular automorphism, then locally on $\abs{S}$, $\EE$ can be obtained
from the construction of Proposition \ref{bundle centralizer}, with $H$ a torus.
If $H$ is a torus, then $\abs{\Bun_H}$ is discrete.

Any $b \in B(G)$ induces a morphism $\Spd \Fpb \to \Bun_G$.  The group $\Gbt$ can be identified with the fiber product
$\Spd \Fpb \times_{\Bun_G} \In(\Bun_G)$.  Let $\chi_{\Gbt}$ denote the composite $\Gbt \to \In(\Bun_G) \to \underline{(G//G)(F)}$.
One can also define a map
$\chi_{G_b} \colon G_b \to (G//G)$
as in Example \ref{isom example}.

\begin{lem}[{\cite[Proposition III.5.1]{fargues-scholze}}] \label{Gbt structure}
The group $\Gbt$ has a slope filtration $\{ \Gbt^{\ge \lambda} \}$
such that
\[ \Gbt \cong \underline{G_b(F)} \ltimes \Gbt^{>0} \,, \]
and for each
$\lambda > 0$, $\Gbt^{\ge \lambda} / \Gbt^{>\lambda}$ is isomorphic to the Banach--Colmez space
associated with the slope $-\lambda$ part of the isocrystal $(\Lie G \otimes_F \breve{F},\operatorname{Ad}(b) \sigma
)$.
\end{lem}

\begin{lem} \label{factor through Gb}
The diagram
\[ \begin{tikzcd}[column sep=huge]
\Gbt \arrow[r] \arrow[rd,"\chi_{\Gbt}"] & \underline{G_b(F)} \arrow[r] \arrow[d,"\chi_{G_b}"] &
\Gbt \arrow[ld,"\chi_{\Gbt}"] \\
& \underline{(G//G)(F)}
\end{tikzcd} \]
commutes.
\end{lem}
\begin{proof}
It is clear from the definitions that the triangle on the right commutes.  To see that the large triangle also commutes, observe that
the composite $\Gbt \to \underline{G_b(F)} \to \Gbt$ is
a limit of conjugation maps, and $\chi_{\Gbt}$ is conjugation
invariant.
\end{proof}

\begin{lem} \label{local homeo}
Let $T$ be a maximal torus of $G_b$, and let $T_{\sr}$
denote the preimage of $(G//G)_{\sr}$ under the map $T \to G//G$ induced
by $\chi_{G_b}$.
Then the induced map
$T(F)_{\sr} \to (G//G)(F)_{\sr}$ is a local homeomorphism.
\end{lem}
\begin{proof}
By Chevalley restriction, the map $T_{\sr} \to (G//G)_{\sr}$ is \'etale.
Then the lemma follows from the inverse function theorem for locally $F$-analytic manifolds \cite[\S II.III.9]{serre-lie-algebras}.
\end{proof}

\begin{proof}[Proof of Theorem \ref{sr open}]
Let $S$ be a perfectoid space, and let $\EE$ be a $G$-bundle over
$X_S$ equipped with a strongly regular endomorphism $g$.
It induces a map $S \to \In(\Bun_G)$.
Let $b \in B(G)$, and let $z = \Spa (K,K^+)$ be a point of $S$ whose image
in $\Bun_G$ is contained in $\Bun_G^b$.
We need to show that some open neighborhood of $z$ also maps to $\Bun_G^b$.

Let $t$ be the image of $g_z$ under the map $\Gbt(K) \to G_b(F)$. 
By Lemma \ref{factor through Gb}, $g_z$ and $t$ determine the same
element of $(G//G)(F)$.
Let $T$ be the centralizer of $t$ in $G_b$.
By Lemma \ref{local homeo}, after replacing $S$ with an open neighborhood
of $z$, we can find a map $\abs{S} \to T(F)_{\sr}$ such that the composite
$\abs{S} \to T(F)_{\sr} \xrightarrow{\chi_{G_b}} (G//G)(F)_{\sr}$ agrees
with the map $\abs{S} \to (G//G)(F)_{\sr}$ induced by $g$.
Let $g'$ the endomorphism of the bundle $\EE_{b,S}$ induced by this map.
By Corollary \ref{sr etale locally conj}, $g$ and $g'$ are \'etale locally conjugate.
By Proposition \ref{bundle centralizer}, we can find a $T$-bundle
$\widetilde{\EE}$ such that $\EE = (\widetilde{\EE} \times \EE_{b,S}) / T$.  Then the map
$S \to \Bun_G$ factors through $\Bun_T$.  Since $\abs{\Bun_T}$ is discrete,
the map $\abs{S} \to \abs{\Bun_G}$ must be locally constant.

\end{proof}

\section{Characteristic classes} \label{char class section}

In this section, we consider the theory of characteristic classes, both in
general and on $\Bun_G$.  We are interested in characteristic classes because
of their connection to trace distributions of representations of $G(F)$.
This connection will be recalled in
section \ref{trace dist section}.
In section \ref{BunG dist}, we will relate characteristic classes
on $\Bun_G$ to distributions on the groups $G_b(F)$ for $b \in B(G)$
(including non-basic $b$).
Sections \ref{BC cohomology} and \ref{cc smooth}
develop some techniques for computing characteristic
classes on $\Bun_G$, and more generally for computing
the behavior of characteristic classes under cohomologically smooth pullback.

From now on, we assume that $F$ is a finite extension of $\Qp$.

\subsection{Definition of characteristic classes}

First, we recall the definition of characteristic classes from \cite[\S 4.3]{hkw}.
In the interest of space, we avoid introducing categories
of cohomological correspondences, and give all definitions directly in terms
of the category $D_{\et}(X,\Lambda)$ defined in \cite[Definition 14.13]{diamonds}.

Let $f \colon X \to S$ be a fine map of decent v-stacks, as defined in \cite[Definitions 1.1 and 1.3]{ghw}.
Let $\ell$ be a prime different from $p$, and
let $\Lambda$ be a $\Zl$-algebra that is killed by a power of $\ell$.
Let $K_{X/S}$ denote the dualizing complex of $f$.

For any object $A$ of $D_{\et}(X,\Lambda)$, let $A^{\vee} = \uRHom(A,K_{X/S})$
denote its dual.  There is a natural morphism
\[ \ev \colon A^{\vee} \otimes A \to K_{X/S} \,, \]
called the \emph{evaluation map}.

Let $\pr_1, \pr_2 \colon X \times_S X \to X$ denote the projection
maps, and let $\Delta \colon X \to X \times_S X$ denote the diagonal.
There is a natural morphism
\begin{equation} \label{dualizable map}
A^{\vee} \boxtimes A \to \uRHom(\pr_1^* A, \pr_2^! A) \,,
\end{equation}
defined as follows.  By the adjunction between $\otimes$ and $\uRHom$,
specifying such a morphism is equivalent to specifying a morphism
$(A^{\vee} \otimes A) \boxtimes A \to \pr_2^! A$,
which is in turn equivalent to specifying a morphism
$\pr_1^* (A^{\vee} \otimes A) \to \uRHom(\pr_2^* A, \pr_2^! A)$.
We take the composite
\[ \pr_1^* (A^{\vee} \otimes A) \xrightarrow{\pr_1^* \ev} \pr_1^* K_{X/S} \to \pr_2^! \Lambda \to \pr_2^! \uRHom(A,A) \xrightarrow{\sim} \uRHom(\pr_2^* A,\pr_2^! A) \,. \]
Here, the second arrow is base change.

\begin{defn}
We say that $A$ is \emph{dualizable} over $S$ if \eqref{dualizable map}
is an isomorphism (cf.~\cite[Proposition 4.3.5]{hkw}).
\end{defn}
If $A$ is dualizable, we define the
\emph{coevaluation map} $\coev \colon \Lambda \to \Delta^! (A^{\vee} \boxtimes A)$ as the composite
\[ \Lambda \to \uRHom(A,A) \isom \uRHom(\Delta^* \pr_1^* A,\Delta^! \pr_2^! A) \isom \Delta^! \uRHom(\pr_1^* A,\pr_2^! A) \isom \Delta^! (A^{\vee} \otimes A) \,, \]
where the last isomorphism is induced by the inverse of \eqref{dualizable map}.

Let $\In_S(X)$ be the inertia stack of $X$, as in Definition \ref{inertia stack def}, and let $\pi_1, \pi_2 \colon \In_S(X) \to X$ denote
the projection maps.

\begin{defn}[{\cite[Definition 4.3.6]{hkw}}]
Let $A \in D_{\et}(X,\Lambda)$ be an object that is dualizable
over $S$.  The \emph{characteristic class} of $A$ over $S$, denoted
$\cc_{X/S}(A)$, is the element of $H^0(\In_S(X),K_{\In_S(X)/S})$
determined by the composite
\[ \Lambda \xrightarrow{\pi_1^* \coev} \pi_1^* \Delta^! (A^{\vee} \boxtimes A) \to \pi_2^! \Delta^* (A^{\vee} \boxtimes A) \xrightarrow{\pi_2^! \ev} K_{\In_S(X)/S} \,, \]
where the second arrow is base change.
\end{defn}

Characteristic classes behave nicely under pushforward by a proper map and pullback by
an open immersion, as the following two propositions demonstrate.

\begin{prop}[{\cite[Lemma 4.3.7]{hkw}}] \label{cc pull}
Let $i \colon U \to X$ be an open immersion of decent v-stacks
fine over $S$.  Then $\In_S(i) \colon \In_S(U) \to \In_S(X)$ is also
an open immersion.  If $A \in D(X,\Lambda)$ is dualizable over $S$, then so
is $i^*A$, and
\[ \cc_{U/S}(i^*A) = \In_S(i)^* \cc_{X/S}(A) \,. \]
\end{prop}
\begin{prop}[{\cite[Corollary 4.3.9]{hkw}}] \label{cc push}
Let $p \colon X \to Y$ be an proper morphism of decent v-stacks
fine over $S$.  Then $\In_S(p) \colon \In_S(X) \to \In_S(Y)$ is also
proper.  If $A \in D(X,\Lambda)$ is dualizable over $S$, then so is $p_! A$,
and
\[ \cc_{X/S}(p_! A) = \In_S(i)_! \cc_{Y/S}(A) \,. \]
\end{prop}

In section \ref{cc smooth}, we will show that the behavior of characteristic classes
under cohomologically smooth pullback can also sometimes be understood explicitly.

\subsection{Geometrization of trace distributions}
\label{trace dist section}

Let $\Lambda$ denote an arbitrary $\Zl$-algebra (not necessarily killed by a power of $\ell$).
Let $C_c(G_b(F),\Lambda)$ denote
the space of locally constant $\Lambda$-valued functions on $G_b(F)$, and let
\[ \Haar(G_b(F),\Lambda) = \Hom_{\Lambda[G_b(F)]}(C_c(G_b(F),\Lambda),\Lambda) \] denote
the space of $\Lambda$-valued Haar measures on $G_b(F)$.
Define
\begin{equation} \label{dist def}
\HDist(G_b(F),\Lambda) = \Hom_{\Lambda}(C_c(G_b(F),\Lambda) \otimes_{\Lambda} \Haar(G_b(F),\Lambda),\Lambda) \,.
\end{equation}

A $\Lambda$-representation $\rho$ of $G_b(F)$ is said to be smooth if the stabilizer of any vector
is open.
A smooth representation $\rho$ is said to be admissible if for every compact open pro-$p$ subgroup
$K \subset G_b(F)$, the derived $K$-invariants of $\rho$ form a perfect complex of $\Lambda$-modules.
Any admissible smooth representation $\rho$ determines a trace distribution
$\trdist(\rho) \in \HDist(G,\Lambda)$.
For any $f \in C_c(G_b(F),\Lambda)$,
the operator $\int_{g \in G_b(F)} f(g) \rho(g)$ has image contained in a finitely presented $\Lambda$-module,
and $\trdist(\rho)(f)$ is defined to be its trace.

If in addition $\Lambda$ is isomorphic to $\mathbb{C}$ and $\rho$ has finite length, then $\trdist(\rho)$ has a
Harish-Chandra character $\Theta_{\rho}$ \cite[Theorem 16.3]{hc-admissible}.
Here, $\Theta_{\rho}$ is a locally constant function from the regular semisimple locus $G_b(F)_{\rs}$ of $G_b(F)$ to $\Lambda$
satisfying
\[ (\trdist \rho)(f) = \int_{g \in G_b(F)_{\rs}} \Theta_{\rho}(g) f(g) \]
for all $f \in C_c(G_b(F),\Lambda) \otimes_{\Lambda} \Haar(G_b(F),\Lambda)$.
(Since $G_b(F)_{\rs}$ is not compact, we choose a isomorphism $\Qlb \cong \mathbb{C}$ to make sense of the integral.
However, since $\Theta_{\rho}$ is uniquely determined by
the values of the integrals for which $f$ has support in a compact subset of $G_b(F)_{\rs}$,
$\Theta_{\rho}$ does not depend on the choice of isomorphism.)
Since the strongly regular locus is dense in the regular semisimple locus, there is no harm
in integrating over $G_b(F)_{\sr}$ instead of $G_b(F)_{\rs}$.

The trace distribution can be reinterpreted in terms of characteristic classes.
Let $D(G_b(F),\Lambda)$ denote the derived category of the category
of $\Lambda$-representations of $G_b(F)$.
For any small v-stack $X$,
Fargues--Scholze define a category $D_{\lis}(X,\Lambda)$
\cite[VII.6.1]{fargues-scholze}.
If $\Lambda$ is killed by a power of $\ell$, then $D_{\lis}(X,\Lambda)$
is equivalent to $D_{\et}(X,\Lambda)$ by \cite[Proposition VII.6.6]{fargues-scholze}.

Until the end of section \ref{BC cohomology},
let $S$ be either $\Spd \Fpb$, or $\Spd C$ for
some algebraically closed nonarchimedean field $C$ of characteristic $p$.
\begin{prop}[{\cite[Proposition VII.7.1]{fargues-scholze}}] \label{rep sheaf equivalence}
There are equivalences
\[D(G_b(F),\Lambda) \cong D_{\lis}([S/G_b(F)],\Lambda) \cong D_{\lis}([S/\Gbt],\Lambda) \,. \]
\end{prop}

\begin{prop}[{\cite[Example 4.2.5]{hkw}}] \label{Gb dist}
Suppose $\Lambda$ is killed by a power of $\ell$.
There is an isomorphism
\[ H^0(\In_S([S/\Gbu],K_{\In_S([S/\Gbu])/S}) \cong \Dist(G_b(F),\Lambda)^{G_b(F)} \,. \]
Here, $G_b(F)$ acts on $\Dist(G_b(F),\Lambda)$ by conjugation.
\end{prop}
\begin{prop} \label{tr dist cc}
Suppose $\Lambda$ is killed by a power of $\ell$.
\begin{enumerate}
\item An object of $D(G_b(F),\Lambda)$ is admissible if and only if the corresponding
object of $D_{\et}([S/\Gbu],\Lambda)$ is dualizable over $S$.
\item Under the isomorphism of Proposition \ref{Gb dist}, if $\rho$ is admissible, then
\[ \cc_{[S/\Gbu]/S}(\rho) = \trdist(\rho) \,. \]
\end{enumerate}
\end{prop}
\begin{proof}
The first claim can be proved in the same way as \cite[Theorem V.7.1]{fargues-scholze}.  The
second claim is \cite[Proposition 4.4.3]{hkw}.
\end{proof}

\subsection{Generalization of distributions to $\Bun_G$}
\label{BunG dist}

We again assume that $\Lambda$ is a $\Zl$-algebra killed by a power of $\ell$.
We are interested in understanding the characteristic classes
of objects in $D_{\et}(\Bun_G,\Lambda)$.
We will focus on the restriction to the strongly regular part of
the inertia stack.
We begin by determining the space in which these characteristic classes live.
\begin{prop}~\label{sr inertia isom}
\begin{enumerate}
\item \label{inertia stack iso}
For each $b \in B(G)$, the map $\underline{G_b(F)} \to \Gbt$ induces an isomorphism
\[ \In_S([S/\underline{G_b(F)}])_{\sr} \xrightarrow{\sim} \In_S([S/\Gbt])_{\sr} \,. \]
\item
For each $b \in B(G)$, there is an isomorphism
\[ H^0(\In_S([S/\Gbt])_{\sr},K_{\In_S([S/\Gbt])_{\sr}/S}) \cong \Dist(G_b(F)_{\sr},\Lambda)^{G_b(F)} \,. \]
\item \label{bun g part}
There is an isomorphism
\[ H^0(\In_S(\Bun_{G,S})_{\sr},K_{\In_S(\Bun_{G,S})_{\sr}/S}) \cong \prod_{b \in B(G)} \Dist(G_b(F)_{\sr},\Lambda)^{G_b(F)} \,. \]
\end{enumerate}
\end{prop}

\begin{lem} \label{conj isom}
The map $c \colon \underline{G_b(F)}_{\sr} \times \Gbt^{>0} \to (\Gbt)_{\sr}$
that sends $(g,h) \mapsto hgh^{-1}$ is an isomorphism of diamonds.
\end{lem}
\begin{proof}
We will show by induction that for each $\lambda \ge 0$,
the map
$c_{\lambda} \colon \underline{G_b(F)}_{\sr} \times \Gbt^{>0}/\Gbt^{>\lambda} \to
\underline{G_b(F)}_{\sr} \times \Gbt^{>0}/\Gbt^{>\lambda}$
defined by $(g,h) \mapsto ghg^{-1}h^{-1}$
is an isomorphism.
The case $\lambda = 0$ is trivial.

Let $\lambda$ be a slope appearing in $\Gbt^{>0}$,
and suppose that $c_{\lambda'}$ is an isomorphism for $\lambda' < \lambda$.
To show that $c_{\lambda}$ is also an isomorphism,
it suffices to prove that for each affinoid perfectoid $Z=\Spa(R,R^+)$,
$c_{\lambda}$ induces an isomorphism on $Z$-points.

Let $X^{\alg}_R$ denote the schematic Fargues--Fontaine
curve over $R$, and let $H$ denote the inner twisting of $G \times X^{\alg}_R$ by $\EE_b$.
The slope filtration on $\EE_b$ induces a filtration on $H$.

By the argument of \cite[Proposition III.5.1]{fargues-scholze},
for $\lambda \le \lambda'$,
\[ (\Gbt^{>0}/\Gbt^{>\lambda'})(Z) = 
(\Gbt^{>0}/\Gbt^{>\lambda})(Z) / (\Gbt^{>\lambda'}/\Gbt^{>\lambda})(Z) \,. \]
So $c_{\lambda'}^{-1}$ admits a lift to an endomorphism of
$(\underline{G_b(F)}_{\sr} \times \Gbt^{>0}/\Gbt^{>\lambda})(Z)$.

For each $g \in G_b(F)$, the endomorphism of $H^{\ge \lambda}/H^{>\lambda}$
induced by $c$ is injective and linear, so it is an isomorphism.
So $c$ induces an automorphism of
$(\underline{G_b(F)} \times \Gbt^{\ge \lambda}/\Gbt^{>\lambda})(Z)$.

Combining these facts, we see that there is a lift of $c_{\lambda'}^{-1}$
that is an inverse of the map of $Z$-points of $c_{\lambda}$.

\end{proof}

\begin{proof}[Proof of Proposition \ref{sr inertia isom}]
The first item follows immediately from Lemma \ref{conj isom}.
The second item follows from the first and Proposition \ref{Gb dist}.
The third item follows from the second and Theorem \ref{sr open}.
\end{proof}

\subsection{Action of automorphisms of Banach--Colmez spaces on
compactly supported cohomology}
\label{BC cohomology}

As mentioned in Lemma \ref{Gbt structure}, $\Gbt$ is a successive
extension of $\Gbu$ by Banach--Colmez spaces.
In order to understand the relation between characteristic classes
on $[S/\Gbu]$ and characteristic classes on $[S/\Gbt]$, it will
be useful to understand how endomorphisms of Banach--Colmez spaces
act on the cohomology of these spaces.

\begin{lem} \label{exact seq action}
Let
\[ 1 \to \mathcal{G}' \to \mathcal{G} \to \mathcal{G}'' \to 1 \]
be an exact sequence of group diamonds that are locally spatial and separated and cohomologically smooth over $S$.
Assume the morphisms are locally
compactifiable of finite $\mathrm{dim.trg.}$

Let $f \colon S \to [S/\mathcal{G}]$ be the quotient map,
and define $f'$, $f''$ similarly.
Let $g$ be an automorphism of $\mathcal{G}$ that respects the exact sequence.

Suppose that $f'_! \Lambda$ and $f''_! \Lambda$ are shifts of
the constant sheaf $\Lambda$, and
that $g$ acts by $\lambda',\lambda'' \in \Lambda^{\times}$ on
$f'_! \Lambda$ and
$f''_! \Lambda$, respectively.
Then $f_! \Lambda$ is also a shift of the constant sheaf $\Lambda$, and $g$
acts by
$\lambda' \lambda''$ on $f_! \Lambda$.
\end{lem}
\begin{proof}
Consider the diagram
\begin{equation} \label{cc coh base change} \begin{tikzcd}
S \arrow[r] & {[S/\mathcal{G}']} \arrow[r] \arrow[d] & {[S/\mathcal{G}]} \arrow[d] \\
& S \arrow[r] & {[S/\mathcal{G}'']}
\end{tikzcd} \,.
\end{equation}
The square is cartesian, so we may apply base change.
\end{proof}

\begin{lem} \label{pushforward constant sheaf}
Let $\EE$ be an object in $D^b(\Coh(X_C))$ such that $\mathcal{H}^0(\EE)$ has
positive slopes, $\mathcal{H}^{-1}(\EE)$ has negative slopes, and all other cohomology 
groups are zero.
Let $\mathcal{BC}(\EE)$ denote the corresponding Banach--Colmez space.
Let $f \colon \Spd C \to [\Spd C/\mathcal{BC}(\EE)]$
be the quotient map.  Then
$f_! \Lambda \cong \Lambda[-2 \deg \EE]$.
\end{lem}
\begin{proof}
First, consider the case where $\EE$ is a vector bundle with positive slopes.
Let $\tilde{f} \colon \mathcal{BC}(\EE) \to \Spd C$ be the base
change of $f$.
By \cite[Proposition 4.8]{hansen-moduli},
$\tilde{f}_! \Lambda \cong \Lambda[-2 \dim \EE]$.
To show that $f_! \Lambda \cong \Lambda[-2 \dim \EE]$, we need to show
that translation by $\mathcal{BC}(\EE)$ induces the identity on
$\tilde{f}_! \Lambda$.  Indeed, $\mathcal{BC}(\EE)$ is connected, while
$\tilde{f}_! \Lambda$ is discrete.

In general, we can find an exact triangle
$\mathcal{F} \to \mathcal{G} \to \EE$, where $\mathcal{F}$ and $\mathcal{G}$ have
positive slopes.  There is a corresponding exact sequence
$0 \to \BC(\mathcal{F}) \to \BC(\mathcal{G}) \to \BC(\EE) \to 0$.
By applying this lemma to $\mathcal{F}$ and $\mathcal{G}$ and considering
the diagram \eqref{cc coh base change}, we see that
the pullback of $f_! \Lambda$ to $[\Spd C/\mathcal{BC}(\mathcal{G})]$
is isomorphic to $\Lambda[-2 \deg \EE]$.
So $f_! \Lambda$ must also be isomorphic to $\Lambda[-2 \deg \EE]$.
\end{proof}

\begin{defn} \label{det def}
Let $\EE$ be as in Lemma \ref{pushforward constant sheaf}, and let
$g$ be an automorphism of $\EE$.  Define $\det g \in F^{\times}$ as follows.
If $\EE$ is a vector bundle with positive slopes, then
$\det g$ is the image of $g$ in $\Aut(\wedge^{\rk \EE} \EE) \cong F^{\times}$.  If $\EE[-1]$ is a vector bundle with negative slopes,
then $\det g$ is the image of $g^{-1}$ in $\Aut(\wedge^{\rk \EE[-1]} \EE[-1])$.
If $\EE$ is torsion, then $\det g = 1$.  In general,
$\det g$ is the product of the determinants of $g$ acting on
the graded pieces of the slope filtration of $\EE$.
\end{defn}

\begin{lem} \label{cohomology action}
Let $\EE$, $\mathcal{BC}(\EE)$, $f$ be as in Lemma \ref{pushforward constant sheaf}.
Let $g$ be an automorphism of $\EE$.
Then the action of $g$ on $f_! \Lambda$
is multiplication by $\left|\det g\right|^{-1}$,
where the norm on $F$ is defined so that $|p|=p^{-[F:\Qp]}$.)
\end{lem}
\begin{proof}

We will reduce to a special case using Lemma \ref{exact seq action}.
Lemma \ref{pushforward constant sheaf} guarantees that the conditions
of Lemma \ref{exact seq action} are met.

Using the slope filtration, we reduce to the case where $\EE$ is semistable.
First, suppose $\EE$ has infinite slope.
By d\'evissage, we can reduce to the case where
$\BC(\EE)$ is a product of copies of $\mathbb{G}_a^{\lozenge}$.  Then the lemma is satisfied, since
any $g$ has determinant $1$ and acts trivially
on the compactly supported cohomology of $\BC(\EE)$.

If $\EE$ has finite positive (resp.~negative) slope,
then by using an exact triangle of the form
$\EE(-n) \to \EE \to \EE/\EE(-n)$
(resp.~$(\EE[-1](n) \to \EE[-1](n)/\EE[-1] \to \EE$),
we reduce to the case where the slope of $\EE$ is positive and
$\le [F:\Qp]$.
Using Jordan normal form, we further reduce to the case where $g$ acts by a scalar (and $F$ is replaced by a finite extension).
In that case, $\BC(\EE)$ is the universal cover of the generic fiber of a
formal $\OO_F$-module $E$ over $\OO_C$.
Since $\OO_F \setminus \{0\}$ generates $F^{\times}$, we may also assume
that the scalar is in $\OO_F$, so $g$ induces an endomorphism of $E$.
Using \cite[Theorem 7.3.4]{huber-etale}, we see that
$g$ must act on $f_! \Lambda$ by the degree
of this endomorphism, which is $\left|\det g\right|^{-1}$.
\end{proof}

Let $b \in B(G)$.  We can also compute the action of endomorphisms
of $\Gbt^{>0}$ on its compactly supported cohomology.

Let $\nu_{b}$ be the Newton point of $b$.
Recall from \cite[\S III.5.1.1]{fargues-scholze} that
$\nu_{b}$ and $-\nu_b$ determine a pair of opposite parabolic
subgroups $P_b^+,P_b^-$ of the quasisplit inner form of $G$, and
$G_b$ is an inner form of $M_b = P_b^+ \cap P_b^-$.
Over some finite extension $F'$ of $F$, $(G_b)_{F'}$ becomes conjugate to
$(M_b)_{F'}$; use this conjugation to identify the two groups.
For $g \in G_b(F)$, let $D_b^-(g)$ be defined by
\begin{equation} \label{Db def}
D_b^-(g) = \det(1-\Ad g|\Lie ((P_b^-)_{F'})/\Lie ((M_b)_{F'})) \,.
\end{equation}
Then $D_b^-(g)$ lies in $F$ since $\Gal(F'/F)$ acts on $g$ by conjugation.
It does not depend on $F'$ or the identification of $(G_b)_{F'}$ and
$(M_b)_{F'}$ as conjugates.

\begin{lem} \label{unipotent factor}
Let $f \colon (\Gbt^{>0})_S \to S$ be the structure map.
Then $f_! \Lambda$ is a shift of $\Lambda$.  If $m \in G_b(F)$, then
the map $\Gbt^{>0} \to \Gbt^{>0}$ defined by $n \mapsto m^{-1} n^{-1} m n$
acts on $f_! \Lambda$ by $\abs{D_b^-(m)}^{-1}$.
\end{lem}
\begin{proof}
It suffices to prove the lemma in the case
$S=\Spd C$; the case $S=\Spd \Fpb$
then follows by base change.

The group $\Gbt^{>0}$ has a filtration whose graded pieces
$\Gbt^{\ge \lambda}/\Gbt^{>\lambda}$
are Banach--Colmez spaces.  Let $\hat{f} \colon S \to [S/(\Gbt^{>0})_S]$
denote the quotient map; then repeated applications
of Lemmas \ref{exact seq action} and \ref{pushforward constant sheaf}
imply that $\hat{f}_! \Lambda$ is a shift of $\Lambda$.
By base change, $f_! \Lambda$ is also a shift of $\Lambda$.
To see that the map $n \mapsto m^{-1} n^{-1} m n$
acts by multiplication by $|D_b^-(m)|^{-1}$, apply Lemma \ref{cohomology action}
to the quotients $\Gbt^{\ge \lambda}/\Gbt^{>\lambda}$.
\end{proof}

\subsection{Characteristic classes under cohomologically smooth pullback}
\label{cc smooth}

We will prove some results regarding the behavior of characteristic
classes under cohomologically smooth pullback.

This subsection is rather technical.  The reader is encouraged to read
Propositions \ref{cc formula simple} and \ref{BunG case} and
Example \ref{gl2 example} for motivation.

We now let $S$ denote an arbitrary decent v-stack.
Let $f \colon X \to Y$ be a cohomologically smooth morphism
of decent v-stacks fine over $S$.
\begin{lem}
Let $A \in D_{\et}(Y,\Lambda)$ be an object that is
dualizable over $S$.  Then $f^*A$ is also dualizable over $S$.
\end{lem}
\begin{proof}
Let $p_1,p_2$ (resp.~$\tilde{p}_1,\tilde{p}_2$) denote the projection maps
$Y \times_S Y \to Y$ (resp.~$X \times_S X \to X$).
We have a commutative diagram
\[ \begin{tikzcd}
\uRHom(f^* A,K_{X/S}) \boxtimes_S f^*A \arrow[r] \arrow[d,"\sim"] &
\uRHom(\tilde{p}_1^* f^* A,\tilde{p}_2^! f^* A) \arrow[d,"\sim"] \\
(f \times \id)^! (\id \times f)^* (\uRHom(A,K_{Y/S}) \boxtimes_S A) \arrow[r] &
(f \times \id)^! (\id \times f)^* \uRHom(p_1^*A,p_2^! A) \,.
\end{tikzcd} \,. \]
If the bottom horizontal arrow is an isomorphism,
then the top horizontal arrow must be an
isomorphism as well.
\end{proof}

\begin{prop} \label{cc change of vars}
The following diagram is commutative with cartesian squares.
\begin{equation} \label{change of vars diagram} \begin{tikzcd}
\In_S(X) \arrow[r,"\tilde{\pi}_2"] \arrow[d,"\tilde{\Delta}_f"]
\arrow[ddd,bend right=60,swap,"\In_S(f)"] &
X \arrow[d,"\Delta_f"] \\
\In_S(Y) \times_Y X \arrow[dd,"\tilde{f}"] \arrow[r,"\hat{\pi}_2"] &
X \times_Y X \arrow[r,"p_2"] \arrow[d,"p_1"] &
X \arrow[dd,"f"] \\
& X \arrow[rd,"f"] \\
\In_S(Y) \arrow[rr,"\pi_2"] & & Y
\end{tikzcd} \end{equation}
Consider the map
\[
\Lambda_{\In_S(X)} \to \In_S(f)^! \Lambda_{\In_S(Y)}
\]
defined by applying the composite natural transformation
\begin{equation} \label{pullback morphism}
\In_S(f)^* \pi_2^* \isom
\tilde{\pi}_2^* \Delta_f^* p_1^* f^* \isom
\tilde{\pi}_2^* \Delta_f^! p_1^! f^* \to
\tilde{\Delta}_f^! \hat{\pi}_2^* p_1^! f^* \isom
\tilde{\Delta}_f^! \hat{\pi}^* p_2^* f^! \isom
\tilde{\Delta}_f^! \tilde{f}^! \pi_2^*
\end{equation}
to $\Lambda_Y$.
Here,
the third arrow is base change, and the fourth and fifth arrows are
smooth base change.
For any $A \in D(X,\Lambda)$ such that $(Y,A)$ is dualizable over $S$,
the morphism
\begin{equation} \label{cc relation}
\Lambda_{\In_S(X)} \xrightarrow{\eqref{pullback morphism}} \In_S(f)^! \Lambda_{\In_S(Y)} \xrightarrow{\In_S(f)^! \cc_{Y/S}(A)} \In_S(f)^! K_{\In_S{Y}/S} \isom K_{\In_S{X}/S}
\end{equation}
is $\cc_{X/S}(f^*A)$.
\end{prop}
\begin{proof}
By considering the diagram
\[ \begin{tikzcd}
& & & \In_S(Y) \arrow[rd] \arrow[dddd,bend left] \\
\In_S(X) \arrow[r] \arrow[d] &
\In_S(Y) \times_Y X \arrow[r] \arrow[d] \arrow[urr] &
X \arrow[rr,crossing over] \arrow[d] & & Y \arrow[dd] \\
X \arrow[r] & X \times_Y X \arrow[r] \arrow[d] &
X \times_S X \arrow[r] \arrow[d] & X \arrow[dd] \\
& X \arrow[r] \arrow[rrd] & X \times_S Y \arrow[rr,crossing over] & & Y \times_S Y \\
& & & Y \arrow[ur]
\end{tikzcd} \,, \]
we see that the morphism \eqref{cc relation}
is computed by the cohomological correspondence
\[ \begin{tikzcd} & X \arrow[rd] \arrow[ld] & & X \arrow[rd] \arrow[ld] \\
(X,\Lambda_X) & & (X \times_S X,f^! A^{\vee} \boxtimes f^* A) & & (X,K_{X/S})
\end{tikzcd}
\,,
\]
where the left half is obtained by applying $f^*$ to $\coev_Y$ and then applying
smooth base change,
and the right half is obtained by applying $f^!$ to $\ev_Y$ and then applying
smooth base change.  It is then clear that the right half of the
correspondence is $\ev_X$.  The following lemma verifies that left half
is $\coev_X$.
\end{proof}

\begin{lem} \label{coev compatibility}
Consider the coevaluation maps
\[ \operatorname{coev} \colon \Lambda \to \Delta_{Y/S}^! (A^{\vee} \boxtimes A) \]
\[ \widetilde{\operatorname{coev}} \colon \Lambda \to \Delta_{X/S}^! (f^! A^{\vee} \boxtimes f^* A) \,. \]
The following diagram commutes.
\[ \begin{tikzcd}
& \Lambda \arrow[ld,"\widetilde{\coev}"] \arrow[rd,"f^* \coev"] \\
\Delta_{X/S}^! (f^! A^{\vee} \boxtimes f^* A) \arrow[r,"\sim"] &
\Delta_{X/S}^! (f \times \id)^! (\id \times f)^* (A^{\vee} \boxtimes A) \arrow[r,"\sim"] &
f^* \Delta_{Y/S}^!(A^{\vee} \boxtimes A)
\end{tikzcd} \]
Here, the lower right arrow is smooth base change.
\end{lem}
\begin{proof}
This follows from the commutativity of the diagram
\[ \begin{tikzcd}
\Lambda \arrow[d,"\sim"] \arrow[r] &
\uRHom(f^*A,f^*A) \arrow[d,"\sim"] \arrow[r,"\sim"] &
\Delta_{X/S}^!\uRHom(\tilde{p}_1^* f^* A,\tilde{p}_2^! f^* A) \arrow[d,"\sim"] \arrow[r,"\sim"] &
\Delta_{X/S}^! (f^! A^{\vee} \boxtimes f^* A) \arrow[d,"\sim"] \\
\Lambda \arrow[r] &
f^* \uRHom(A,A) \arrow[r,"\sim"] &
f^* \Delta_{Y/S}^! \uRHom(p_1^* A, p_2^! A) \arrow[r,"\sim"] &
f^* \Delta_{Y/S}^! (A^{\vee} \boxtimes A)
\end{tikzcd} \,. \]
\end{proof}

In \eqref{pullback morphism}, the only map that is not an
isomorphism is base change along $\Delta_f$.  The following proposition will allow
us to compute the base change explicitly in many cases of interest.

\begin{prop} \label{semidirect inclusion 1}
Let $G = M \ltimes N$ be a semi-direct product of group diamonds over $S$, with $N$ cohomologically smooth.
Let $M_{\reg}$ be an open subdiamond of $M$, and let $G_{\reg} = M_{\reg} N$.
Suppose that the morphism
$c \colon [G_{\reg}//M] \to [G_{\reg}//M]$ sending $mn \mapsto n^{-1} mn$
is an isomorphism.

Let $f \colon [S/M] \to [S/G]$ denote the map induced by the inclusion
$M \to G$.  In the notation of Proposition \ref{cc change of vars}, we can identify
$\In_S(Y)$ with $[G//G]$, $\In_S(Y) \times_Y X$ with $[G//M]$,
and $\In_S(X)$ with $[M//M]$.

The restriction of \eqref{pullback morphism} to $[M_{\reg}//M]$ is given by
\begin{equation} \label{pullback case 1}
\tilde{\Delta}_{f,\reg}^* \tilde{f}^*_{\reg} \Lambda \isom
\tilde{\Delta}_{f,\reg}^! \tilde{f}^!_{\reg} \Lambda \isom
\tilde{\Delta}_{f,\reg}^! \hat{\pi}_{2,\reg}^* p_2^* f^! \Lambda \isom
\tilde{\Delta}_{f,\reg}^! c_* c^* \hat{\pi}_{2,\reg}^* p_2^* f^! \Lambda \isom
\tilde{\Delta}_{f,\reg}^! \tilde{f}^!_{\reg} \Lambda \,.
\end{equation}
Here, the subscript $_{\reg}$ indicates restriction from $[M//M]$ to $[M_{\reg}//M]$,
the first arrow uses the fact that $\tilde{f}_{\reg} \tilde{\Delta}_{f,\reg}$
is an isomorphism,
the second and fourth arrows are smooth base change,
and the third arrow is induced by the natural transformation $\id \isom c_* c^*$.
\end{prop}
\begin{proof}
After restricting left column of the diagram \eqref{change of vars diagram} to $[G_{\reg}//G]$,
we obtain the following diagram.
\begin{equation} \label{case one diagram} \begin{tikzcd}
{[M_{\reg}//M]} \arrow[r] \arrow[d] &
{[S/M]} \arrow[d] \\
{[G_{\reg}//M]} \arrow[dd] \arrow[r,"mn \mapsto n"] &
{[N//M]} \arrow[r] \arrow[d] &
{[S/M]} \arrow[dd] \\
& {[S/M]} \arrow[rd] \\
{[G_{\reg}//G]} \arrow[rr] & & {[S/G]}
\end{tikzcd} \end{equation}
We claim that there is a commutative diagram
\begin{equation} \label{case one second diagram} \begin{tikzcd}[column sep=large]
{[G_{\reg}//M]} \arrow[r,"n^{-1}mn \mapsto n"] \arrow[dd] & {[N//M]} \arrow[r] \arrow[d] & {[S/M]} \arrow[dd] \\
& {[S/M]} \arrow[rd] \\
{[G_{\reg}//G]} \arrow[rr] \arrow[ur] & & {[S/G]}
\end{tikzcd} \,, \end{equation}
where the 2-morphisms filling in the outer and upper right squares match the ones in \eqref{case one diagram}.  Indeed, the 2-morphism in the 
upper right square is induced by the inclusion map $N\to G$,
the 2-morphism in the large square is induced by the constant
map $G_{\reg} \to G$ sending everything to the identity, we can choose
the 2-morphism in the upper left square to be induced by the constant
map $G_{\reg} \to M$ sending everything to the identity, and we can
choose the 2-morphism in the bottom triangle to be induced by the
map $G_{\reg} \to G$ sending $n^{-1} m n \mapsto n^{-1}$.
The composition of the 2-morphisms filling in the smaller squares
and triangle matches the 2-morphism filling in the outer square.

The two arrows $[G_{\reg}//M] \to [N//M]$ are related by composition with $c$.
Then a diagram chase shows that \eqref{pullback case 1} agrees with \eqref{pullback morphism}.
\end{proof}
\begin{prop} \label{semidirect inclusion 2}
Let $G$, $M$, $G_{\reg}$, $M_{\reg}$ be as in Proposition \ref{semidirect inclusion 1}.  Let $c' \colon [(M_{\reg} \times N)//G] \to [(M_{\reg} \times N)//G]$ be given by $(m,n) \mapsto (m,m^{-1} n^{-1} m n)$.
Let $f \colon [N//M] \to [S/M]$ be the obvious map.

In the notation of Proposition \ref{cc change of vars}, we can identify
$\In_S(Y)$ with $[M//M]$, $\In_S(Y) \times_Y X$ with $[(M \times N)//M]$,
and $\In_S(X) \times_{\In_S(Y)} [M_{\reg}//M]$ with $[M_{\reg}//M]$.

The restriction of \eqref{pullback morphism} to $[M_{\reg}//M]$ is given by
a map analogous to \eqref{pullback case 1}, with $c'$ in place of $c$.
\end{prop}
\begin{proof}
After pulling back the left column of \eqref{change of vars diagram}
along $[M_{\reg}//M] \to [M//M]$, the diagram becomes
\begin{equation} \label{case two large} \begin{tikzcd}[column sep=8em]
{[M_{\reg}//M]} \arrow[r,] \arrow[d] &
{[N//M]} \arrow[d] \\
{[(M_{\reg} \times N)//M]} \arrow[dd] \arrow[r,"{(m,n) \mapsto (m^{-1}nm,n)}"] &
{[(N\times N)//M]} \arrow[r,"{(n_1,n_2) \mapsto n_2}"] \arrow[d] &
{[N//M]} \arrow[dd] \\
& {[N//M]} \arrow[rd] \\
{[M_{\reg}//M]} \arrow[rr] & & {[S/M]}
\end{tikzcd} \,. \end{equation}
The top left square can be extended to form the following diagram.
\begin{equation} \label{case two top} \begin{tikzcd}[column sep=8em]
{[M_{\reg}//M]} \arrow[r] \arrow[d] & {[N//M]} \arrow[r] \arrow[d] & {[S/M]} \arrow[d] \\
{[(M_{\reg} \times N)//M]} \arrow[r,"{(m,n) \mapsto (m^{-1}nm,n)}"] & {[(N \times N)//M]} \arrow[r,"{(n_1,n_2) \mapsto n_1^{-1} n_2}"] & {[N//M]}
\end{tikzcd} \end{equation}
The two arrows $[(M_{\reg} \times N)//M] \to [N//M]$ are related by
composition with $c'$.
Then a diagram chase shows that the analogue of \eqref{pullback case 1} agrees with \eqref{pullback morphism}.
\end{proof}

The remainder of this section is devoted to applications
of Propositions \ref{semidirect inclusion 1} and
\ref{semidirect inclusion 2}.
For these applications, let $S = \Spd \Fpb$ or $\Spd C$
for an algebraically closed nonarchimedean field $C$.

\begin{cor} \label{Gbt case}
Let $f \colon [S/\Gbu] \to [S/\Gbt]$ be the map induced by the
inclusion $\Gbu \hookrightarrow \Gbt$.  Let $A$ be a dualizable object in
$D_{\et}([S/\Gbu],\Lambda)$.  Then under the isomorphism of
Proposition \ref{sr inertia isom}(\ref{inertia stack iso}),
\[ \cc_{[S/\Gbu]/S}(A)|_\sr = \abs{D_b^-}^{-1} \cc_{[S/\Gbt]/S}(f^*A)_\sr \,. \]
\end{cor}
\begin{proof}
We will apply Proposition \ref{semidirect inclusion 1},
with $G = \Gbt$, $M = \Gbu$, $M_{\reg}=\Gbu_{\sr}$, $N = \Gbt^{>0}$.
Using Lemma \ref{unipotent factor} and smooth base change
with respect to the cartesian squares
\begin{equation} \begin{tikzcd} \label{two base changes}
N \arrow[d] & G_{\reg} \arrow[l] \arrow[r] \arrow[d] & {[G_\reg//M]} \arrow[d] \\
S & M_{\reg} \arrow[l] \arrow[r] & {[G_{\reg}//G]}
\end{tikzcd} \,, \end{equation}
we see that $\tilde{f}_{\reg !} \hat{\pi}_{2,\reg}^* p_2^* f^! \Lambda$
is locally free of rank one, and $c$ acts on it by $\abs{D_b^-}^{-1}$.
In other words, if we compose the map
\[\tilde{f}_{\reg!} \hat{\pi}_{2,\reg}^* p_2^* f^! \Lambda \to
\tilde{f}_{\reg!} c_* c^* \hat{\pi}_{2,\reg}^* p_2^* f^! \Lambda \]
induced by $\id \to c_* c^*$ with the map
\[\tilde{f}_{\reg!} c_* c^* \hat{\pi}_{2,\reg}^* p_2^* f^! \Lambda \to
\tilde{f}_{\reg!} \hat{\pi}_{2,\reg}^* p_2^* f^!  \Lambda \]
induced by $\tilde{f}_{\reg!} c_* \isom \tilde{f}_{\reg}$,
$c^* \hat{\pi}^* p_2^* \isom \hat{\pi}^* p_2^*$,
the result is multiplication by $|D_b^-|^{-1}$.

We claim that the map 
\begin{equation} \label{upper lower iso}
\hat{\pi}_{2,\reg}^* p_2^* f^! \Lambda \to
\tilde{f}_{\reg}^!\tilde{f}_{\reg!} \hat{\pi}_{2,\reg}^* p_2^* f^! \Lambda \end{equation}
induced by $\id \to \tilde{f}_{\reg}^! \tilde{f}_{\reg!}$
is an isomorphism.  Indeed, by smooth base change, it is enough
to check that $f^! \Lambda \to f^! f_! f^! \Lambda$ is an
isomorphism.  This map has a left inverse $f^! f_! f^! \Lambda \to f^! \Lambda$
induced by the natural transformation $f_! f^! \to \id$.  By
repeated applications of
\cite[Proposition 4.8(i)]{hansen-moduli} and smooth base change,
$f_! f^! \Lambda \to \Lambda$ is an isomorphism,
so $f^! f_! f^! \Lambda \to f^! \Lambda$ is an isomorphism as well.

Using the isomorphism \eqref{upper lower iso} along with the fact
that $\tilde{f}_{\reg} \tilde{\Delta}_{f,\reg}$ is an isomorphism,
we see that $c$ also acts by multiplication by $|D_b^-|^{-1}$ on
$\tilde{\Delta}_{f,\reg}^! \hat{\pi}_{2,\reg}^* p_2^* f^! \Lambda$.
Now applying Proposition \ref{semidirect inclusion 1} yields
the desired result.
\end{proof}
\begin{prop} \label{cc formula simple}
Let $\rho$ be a representation of $G_b(F)$, and let $i_b \colon [S/\Gbu] \to \Bun_{G,S}$
be the usual map.
\begin{enumerate}
\item Under the isomorphism of Proposition
\ref{sr inertia isom}(\ref{inertia stack iso}),
\[ \abs{D_b^-}^{-1} \cc_{\Bun_{G,S}/S}(i_{b*} \rho)|_{\In_S(\Bun_{G,S}^b)_{\sr}} = \cc_{[S/\Gbu]/S}(\rho)_{\sr} \,. \]
\item If $b' \in B(G)$ is not a specialization of $b$, then
\[ \cc_{\Bun_{G,S}/S}(i_{b*} \rho)|_{\In_S(\Bun_{G,S}^{b'})_{\sr}} = 0 \,. \]
\end{enumerate}
\end{prop}
\begin{proof}
The first claim follows from Corollary \ref{Gbt case} and Proposition \ref{cc push}.

The to prove the second claim, observe that $i_b$ is a composite of the open
immersion $\Bun_{G,S}^b \to \overline{\Bun_{G,S}^b}$ and the closed immersion
$\overline{\Bun_{G,S}^b} \to \Bun_{G,S}$.
So this claim follows from Propositions \ref{cc pull} and \ref{cc push}.
\end{proof}
If $b'$ is not a specialization of $b$, then the situation is more
complicated, as we will see in Example \ref{gl2 example}.
\begin{prop} \label{BunG case}
Let $f \colon \mathcal{M}_{b,S} \to \Bun_{G,S}$ be the chart defined in
\cite[Definition V.3.2]{fargues-scholze}, and let
$A$ be a dualizable object in $D_{\et}(\Bun_{G,S},\Lambda)$.
The base change of $\In_S(\mathcal{M}_{b,S}) \to \In_S(\Bun_{G,S})$ to $\In_S(\Bun_{G,S}^b)_{\sr}$
is an isomorphism.
Under this isomorphism,
\[ |D_b^-|^{-1} \cc_{\Bun_{G,S}/S}(A)_{b,\sr} = \cc_{\mathcal{M}_{b,S}/S}(f^*A)_{b,\sr}\,. \]
\end{prop}
\begin{proof}
We can identify $(\Gbt)_S$ with a substack of $\Bun_{G,S}$.
By Theorem \ref{sr open}, $[(\Gbt)_{\sr}//\Gbt]_S$
is identified with an open substack of $\In_S(\Bun_G)$.  Then
same arguments as in Proposition \ref{semidirect inclusion 1}
and Corollary \ref{Gbt case} apply.  We use smooth base change
with respect to the cartesian square
\[ \begin{tikzcd}
{[S/\Gbu]} \arrow[r] \arrow[d] & \mathcal{M}_{b,S} \arrow[d] \\
{[S/\Gbt]} \arrow[r] & \Bun_{G,S}
\end{tikzcd} \]
to show that the calculations of Corollary \ref{Gbt case}
still apply in this setting.
\end{proof}
\begin{cor} \label{banach colmez case}
In the situation of Proposition \ref{semidirect inclusion 2},
suppose that $N = \BC(\EE)$, where $\EE$ is as in Lemma \ref{pushforward constant sheaf}.
Then the map \eqref{pullback morphism} is given by multiplication by
the function $M \to F^{\times}$ sending $m \mapsto \left|\det(m-1)\right|^{-1}$.
Here, $\det$ denotes the determinant of the action on $\EE$,
as defined in Definition \ref{det def}.
\end{cor}
\begin{proof}
This follows from Lemma \ref{cohomology action} and Proposition \ref{semidirect inclusion 2} in the same way
that Corollary \ref{Gbt case} follows from Lemma \ref{unipotent factor} and Proposition \ref{semidirect inclusion 1}.
\end{proof}

\begin{ex} \label{gl2 example}
Let $G = \GL_2$. Let $b, b' \in B(G)$ correspond to the
vector bundles $\OO(1/2)$ and $\OO \oplus \OO(1)$, respectively,
on the Fargues--Fontaine curve.
Then $G_b(F)$ is isomorphic to the group of units in a nontrivial quaternion
algebra over $F$, and $G_{b'}(F)$ is isomorphic to $(F^{\times})^2$.

Let $i_{b*} \colon [S/\Gbu] \to \Bun_{G,S}$ denote the inclusion map.
We will use Proposition \ref{BunG case} and Corollary \ref{banach colmez case}
to compute the strongly regular $b'$-part of the characteristic class of
$i_{b*} \Lambda$.

Let $\pi_{b'} \colon \mathcal{M}_{b',S} \to \Bun_{G,S}$ be the chart defined
in \cite[Definition V.3.2]{fargues-scholze}.
By Proposition \ref{BunG case},
\[ \cc_{\Bun_{G,S}/S}(i_{b*}\Lambda)_{b',\sr} = \abs{1-t_1/t_2} \cc_{\mathcal{M}_{b',S}}(\pi_{b'}^* i_{b*} \Lambda)_{b',\sr} \,. \]
Here, $t_1,t_2$ are coordinates on $G_{b'}(F) \cong (F^{\times})^2$.
There is an exact triangle
\[ \Lambda \to f^* i_{b*} \Lambda \to j_! \Lambda[-1] \,, \]
where $j \colon [S/\Gbu] \hookrightarrow \mathcal{M}_{b',S}$
is the inclusion of the closed stratum of $\mathcal{M}_{b',S}$ in $\mathcal{M}_{b',S}$.
So
\[ \cc_{\mathcal{M}_{b',S}}(\pi_{b'}^* i_{b*} \Lambda) = \cc_{\mathcal{M}_{b',S}}(\Lambda) - \cc_{\mathcal{M}_{b',S}}(j_! \Lambda) \,. \]
By Proposition \ref{cc push},
\[ \cc_{\mathcal{M}_{b'}}(j_! \Lambda) = \In(j)_! \cc_{[S/G_{b'}(F)]}(\Lambda) \,. \]
The characteristic class $\cc_{[S/\Gbpu]/S}(\Lambda)$ is the distribution that
has Harish-Chandra character $1$.  We will abuse notation slightly and write
\[ \cc_{[S/\Gbpu]/S}(\Lambda) = 1 \,. \]
Recall that $\mathcal{M}_{b',S}$ is isomorphic to $(\BC(\OO(-1))/\Gbpu)_S$.
By Proposition \ref{semidirect inclusion 2} applied to
$\mathcal{M}_{b',S} \to [S/\Gbu]$,
\[ \cc_{\mathcal{M}_{b',S}/S}(\Lambda)_{b',\sr} = \abs{1-t_2/t_1} \cc_{[S/\Gbu]/S}(\Lambda)_{b',\sr} = |1-t_2/t_1| \,. \]
Putting everything together, we obtain
\[ \cc_{\Bun_{G,S}/S}(i_{b*\Lambda})_{b',\sr} = \abs{1-t_1/t_2}\abs{1-t_2/t_1} - \abs{1-t_1/t_2} \,. \]
\end{ex}
In Example \ref{gl2 example hecke},
we will use this calculation to compute the image of
the compactly supported cohomology of the
Drinfeld upper half plane in 
$\Groth(\GL_2(F))$.

\begin{ex} \label{depth one example}
Let $G$ be a connected reductive group over $F$, let $b \in B(G)$ be basic, and let
$b' \in B(G)$ be an element whose only generalization is $b$.
By a similar calculation to Example \ref{gl2 example},
$\cc_{\Bun_G/S}(i_{b*}\Lambda)_{b',\sr}$
has Harish-Chandra character
\[ \prod_{\alpha \in \Phi^+ \cup \Phi^-} |1-\alpha(t)| - \prod_{\alpha \in \Phi^-} |1-\alpha(t)| \,. \]
Here, $\Phi^+$, $\Phi^-$ denote the set of roots of $G$ that are positive (resp.~negative) with respect to the parabolic $P_b^+$.
\end{ex}
These examples rely on the fact that
$\pi_{b'}^* i_{b*} \Lambda$ is an extension of a constant sheaf by
a shift of a sheaf supported at a single point.  To go further, it will be
necessary to understand the characteristic classes of more
complicated constructible sheaves on $\mathcal{M}_{b'}$.

\section{Hecke correspondences and trace distributions} \label{trace section}

Now, we will use Theorem \ref{sr open}, Proposition \ref{bundle centralizer},
and Proposition \ref{BunG case}
to generalize \cite[Section 6]{hkw}.
We will recall the setup of \cite{hkw} and explain how to generalize it.

\subsection{Local shtuka spaces and Hecke correspondences} \label{hecke section}

Let $F$ be a finite extension of $\Qp$.
Let $G$ be a connected reductive group over $F$,
let $b \in B(G)$, and let $\mu$ be a conjugacy class of cocharacters
$\mathbb{G}_{m,\ol{F}} \to G_{\ol{F}}$.
Let $D(G_b(F),\Qlb)$ denote the derived category of smooth
$\Qlb$-representations of $G_b(F)$.
We will now recall the functor
\[ \rho \mapsto R\Gamma(G,b,\mu)[\rho] \]
from $D(G_b(F),\Qlb)$ to $D(G(F),\Qlb)$
that was mentioned in the introduction.
We also recall its relation to the Hecke operators defined in \cite[\S IX]{fargues-scholze}.

Hecke operators are defined using the diagram of v-stacks
\[ \begin{tikzcd}
 & \Hecke \arrow[ld,"h_1"] \arrow[rd,"h_2"] \arrow[r] & \LHecke \\
\Bun_G & & \Bun_G \times \Div^1
\end{tikzcd} \,. \]
Here, $\Div^1 = (\Spd \breve{F})/\Frob^{\ZZ}$ is the diamond classifying degree $1$ Cartier divisors on the Fargues--Fontaine curve,
$\Hecke$ is the stack classifying modifications of vector bundles
on the Fargues--Fontaine curve that are isomorphisms away from a single point of $\Div^1$,
and $\LHecke$ is the local Hecke stack $[L^+G\backslash LG/L^+G]$.

The stack $\LHecke$ has a stratification by Schubert cells.
For any conjugacy class of cocharacters $\mu$ of $G$,
there is a diagram
\[ \begin{tikzcd}
 & \Hecke_{\le \mu} \arrow[ld,"h_{1,\le \mu}"] \arrow[rd,"h_{2,\le \mu}"] \arrow[r] & \LHecke_{\le \mu} \\
\Bun_G & & \Bun_G \times \Div^1
\end{tikzcd} \,. \]
The maps $h_{1,\le \mu}$ and $h_{2,\le \mu}$ are proper \cite[\S I.2]{fargues-scholze}.

Let $C$ be the completion of the algebraic closure of $\ol{F}$.
The local shtuka space $\Sht_{G,b,\mu,C}$ is the fiber product
\[ \Hecke_{\le \mu} \times_{\Bun_G \times \Bun_G \times \Div^1} \Spd C \,, \]
where the map from $\Spd C$ to the first copy of $\Bun_G$ corresponds to the $G$-bundle
$\EE_{b,C}$, the map from $\Spd C$ to the second copy of $\Bun_G$ corresponds to the trivial
$G$-bundle $\EE_{1,C}$, and the map $\Spd C \to \Div^1$ is induced by the
inclusion $\breve{F} \hookrightarrow C$.

The space $\Sht_{G,b,\mu,C}$ is a locally spatial diamond.
It admits an action of $G(F) \times \Gbt(C)$, and in particular
of the subgroup $G(F) \times G_b(F)$.
For any compact open subgroup $K \subset G(F)$, the quotient
$\Sht_{G,b,\mu,K,C}=\Sht_{G,b,\mu,C}/K$ is also a locally spatial diamond
\cite[\S 23]{berkeley}.

Choose a prime $\ell \ne p$.
For each nonnegative integer $n$, the geometric Satake equivalence \cite[\S VI]{fargues-scholze} associates with
$\mu$ an object $\mathcal{S}_{\mu,n}$ in $D_{\et}(\LHecke,\ZZ/\ell^n \ZZ)$.
(One sometimes uses a different normalization that is defined over $(\ZZ/\ell^n \ZZ)[q^{1/2}]$, where $q$
is the size of the residue field of $F$, in order to trivialize a cyclotomic twist in the Weil group action.
But since we are ignoring the Weil group action, we use a normalization that is always defined over $\ZZ/\ell^n \ZZ$.)
We will also write $\mathcal{S}_{\mu,n}$ for the pullback of this object to $\Sht_{G,b,\mu,K,C}$.
Define
\[ R\Gamma_c(\Sht_{G,b,\mu,K,C},\mathcal{S}_{\mu}) = \varinjlim_U \varprojlim_n R\Gamma_c(U,\mathcal{S}_{\mu,n}) \,, \]
where $U$ runs over quasicompact open subsets of $\Sht_{G,b,\mu,K,C}$.
For any object $\rho$ of $D(G_b(F),\Qlb)$, define
\[ R\Gamma(G,b,\mu)[\rho] = \varinjlim_{K \subset G(F)} \RHom_{G_b(F)}(R\Gamma_c(\Sht_{G,b,\mu,K,C},\mathcal{S}_{\mu})\otimes_{\Zl} \Qlb, \rho) \,. \]
It is an object of $D(G(F),\Qlb)$.
By \cite[Proposition 6.4.5]{hkw}, if $\rho$ is admissible (resp.~of finite length),
then $R\Gamma(G,b,\mu)[\rho]$ is as well.

For any $\Zl$-algebra $\Lambda$, there is a Hecke operator
$T_{\mu} \colon D_{\lis}(\Bun_G,\Lambda) \to D_{\lis}(\Bun_G,\Lambda)^{BW_E}$.
We will also write $T_{\mu}$ for the induced functor
$D_{\lis}(\Bun_{G,C},\Lambda) \to D_{\lis}(\Bun_{G,C},\Lambda)$.
\begin{prop}[{\cite[Proposition 6.4.5]{hkw}}] \label{mant formula}
For any $G$, $b$, $\mu$, and any object $\rho$ of $D(G_b(F),\Qlb)$,
\[ R\Gamma(G,b,\mu)[\rho] \cong i_1^* T_{-\mu} i_{b*} \rho \,. \]
Here, $i_b$ denotes the composite
\[ [\Spd C/\Gbu] \to [\Spd C/\Gbt] \xrightarrow{\sim} \Bun_{G,C}^b \to \Bun_{G,C} \,, \]
and $i_1$ is defined similarly.
\end{prop}

As in \cite{hkw}, we will deduce information about the Hecke operators for $\Lambda=\Qlb$ from the Hecke operators for $\Lambda = \overline{\mathbb{Z}}_{\ell}/\ell^n$.
If $\Lambda$ is killed by a power of $\ell$, then we can identify $D_{\lis}(\Bun_G,\Lambda)$
with $D_{\et}(\Bun_G,\Lambda)$ and $D_{\lis}(\Bun_G,\Lambda)^{BW_E}$ with $D_{\et}(\Bun_G \times [*/\underline{W_E}],\Lambda)$.
The latter can further be identified with a full subcategory of $D_{\et}(\Bun_G \times \Div^1,\Lambda)$.
The induced functor $D_{\et}(\Bun_G,\Lambda) \to D_{\et}(\Bun_G \times \Div^1,\Lambda)$ is given by
\[ A \mapsto h_{2*} (h_1^* A \otimes \mathcal{S}_{\mu,n}) = h_{2,\le \mu *} (h_{1,\le \mu}^* A \otimes \mathcal{S}_{\mu,n}) \,, \]
where $n$ is chosen so that $\ell^n$ kills $\Lambda$.

The problem of computing the trace distribution is therefore reduced
to the problem of determining how the characteristic classes interact
with the operations appearing in Proposition \ref{mant formula}.
The analysis can be summarized as follows.
\begin{itemize}
\item Proposition \ref{cc formula simple} computes 
the restriction of $\cc_{\In_C(\Bun_{G,C})/C}(i_{b*} \rho)$ to some of the
strata of $\In_C(\Bun_{G,C})_{\sr}$.  Our lack of knowledge about what happens
on the remaining strata is the reason that we are unable to extend
Theorem \ref{trace intro} to the entire strongly regular locus.
\item The pullback $h_1^*$ and tensor product with $\mathcal{S}_{\mu}$
are treated in \cite[Proposition 6.4.8]{hkw}.  We make no changes to
this part of their argument.
\item Proposition \ref{cc push} reduces the problem of analyzing
the pushforward $h_{2*}$ to the problem of characterizing the fibers
of $h_2$.  This analysis is performed in Section \ref{fixed point section}.
\item Proposition \ref{cc pull} handles the pullback $i_1^*$.
\end{itemize}
Section \ref{trace proof section} combines all of the ingredients.

\subsection{Fixed points in Hecke correspondences} \label{fixed point section}

Let $\Hecke^{b,*}$ denote preimage of $\Bun_G^b$ in $\Hecke$ under $h_1$, and let $\Hecke^{*,1}$ denote the preimage
of $\Bun_G^b \times \Div^1$ in $\Hecke$ under $h_2$.  Let $\Hecke^{b,1}$ denote the intersection of $\Hecke^{b,*}$ and
$\Hecke^{*,1}$

The stack $\In(\Hecke^{b,1})$ parametrizes triples
$(g,g',\phi)$, where $g \in G(F)$, $g'$ is a point of $\tilde{G}_b$, and
$\phi \colon \EE_1 \dashrightarrow \EE_b$ is a modification, up to
equivalence.  Two triples $(g,g',\phi)$ and $(h,h',\psi)$ are
considered equivalent if there
exist $\gamma \in G(F)$ and a point $\gamma'$ of $\tilde{G}_b$
such that $h=\gamma g \gamma^{-1}, g' = \gamma' h' \gamma^{\prime-1}, \psi = \gamma' \phi \gamma^{-1}$.

Suppose we have a triple $(g,g',\phi)$ such that $g$
is strongly regular, with centralizer $T$.  Then by
Proposition \ref{bundle centralizer}, the modification $\phi$ must come
from a modification of $T$-bundles.
Let $\inv[b](g,g') \in B(T)$ denote the isomorphism class
of the $T$-bundle associated with $\EE_b$.

For any torus $T$ over $F$, let $X_*(T)$ denote its cocharacter lattice.
Both $B(T)$ and $X_*(T)$ have the structure of an abelian group.
There is also an action of $\Gamma = \Gal(\ol{F}/F)$ on $X_*(T)$.
\begin{lem}~ \label{torus modifications}
\begin{enumerate}
\item 
Any homomorphism $\beta_{\Gm} \colon X_*(\Gm) \to B(\Gm)$ extends uniquely
to a natural transformation $\beta \colon X_*(-)_{\Gamma} \to B(-)$
of functors from the category of tori over $F$ to the category
of abelian groups.  If $\beta_{\Gm}$ is an isomorphism,
then $\beta$ is also an isomorphism.
\item Let $\beta \colon X_*(-)_{\Gamma} \to B(-)$ be the natural transformation such that
$\beta_{\Gm}$ sends the identity cocharacter of $X_*(\Gm)$ to the class
of a uniformizer of $\breve{F}$ in $B(\Gm)$.
Let $T$ be a torus over $F$.  Then there is a modification of $T$-bundles
$\EE_1 \dashrightarrow \EE_b$ with parameter $\mu$ if and only
if $\beta_T(\mu)=b$.
\end{enumerate}
\end{lem}
\begin{proof}
The first item follows from \cite[Lemma 2.2 and Proposition 2.3]{kottwitz-isocrystals}.

The rule that sends a cocharacter $\mu$ of $T$ to the isomorphism class
of the $T$-bundle obtained by modifying the trivial $T$-bundle
by $\mu$ is a natural transformation $X_*(-) \to B(-)$.  Since the action of
$\Gamma$ on the Fargues--Fontaine curve $X_{F,C}$ does not change isomorphism classes of vector bundles,
this natural transformation factors through $X_*(-)_{\Gamma} \to B(-)$.
The characterization of the map $X_*(\Gm) \to B(\Gm)$ follows from \cite[\S 8.2.1.1]{fargues-fontaine}.
\end{proof}

\begin{defn} \label{relb def}
Let $\Rel_b$ be
the set of equivalence classes of of triples $(g,g',\lambda)$, where
$g \in G(F)_{\sr}$, $g' \in G_b(F)_{\sr}$ are stably conjugate, and
$\lambda \in X_*(T)$ has the property that its image in
$X_*(T)_{\Gamma} \cong B(T)$ is the same as $\inv[b](g,g')$.
Here, $T$ is the centralizer of $g$ in $G$.
Two triples $(g_1,g_1',\lambda_1)$ and $(g_2,g_2',\lambda_2)$
are considered equivalent if there exist $h \in G(F)$, $h' \in G_b(F)$
satisfying $(g_2,g_2',\lambda_2)=(hg_1h^{-1},h'g_1'h^{\prime-1},h\lambda_1 h^{-1})$.

For any conjugacy class of cocharacters $\mu$ of $G$, let
$\Rel_{b,\mu}$ be the subset of $\Rel_b$ consisting of
triples $(g,g',\lambda)$ for which $\lambda \le \mu$.
\end{defn}
The above definition is slightly different from \cite[Definition 3.2.4]{hkw}, which only required that $\lambda$ and $\inv[b](g,g')$ have the same image in
$\pi_1(G)_{\Gamma}$.  The two definitions agree on the elliptic locus,
by \cite[Theorem 6.2.3]{hkw}.

The following extends \cite[Corollary 6.2.4]{hkw}.
\begin{lem}
The inclusions
\[ \In_C(\Hecke^{b,*}_{G,\le \mu,C})_{\sr} \hookleftarrow
\In_C(\Hecke^{b,1}_{G,\le \mu,C})_{\sr} \hookrightarrow \In_C(\Hecke^{*,1}_{G,\le \mu,C})_{\sr} \]
are open and closed immersions.

There is a homeomorphism
\[ \left| \In_C\left(\Hecke^{b,1}_{G,\le\mu,C} \right)_{\sr} \right| \cong
\Rel_{b,\mu} \,. \]
\end{lem}
\begin{proof}
The first claim follows from Theorem \ref{sr open}, and the second follows
from Proposition \ref{bundle centralizer} and Lemma \ref{torus modifications}.
\end{proof}

\subsection{Computing trace distributions} \label{trace proof section}

\begin{lem}
Let $T$ be a maximal torus of $G$.  Then
the map 
$[T(F)_{\sr}//T(F)] \to [G(F)_{\sr}//G(F)]$ induced
by the inclusion $T \hookrightarrow G$
identifies $[T(F)_{\sr}//T(F)]$ with an open and closed substack of
$[G(F)_{\sr}//G(F)]$.

Hence for any $\Zl$-algebra $\Lambda$,
$\Dist(T(F)_{\sr},\Lambda)$
is naturally identified with
a direct summand of $\Dist(G(F)_{\sr},\Lambda)^{G(F)}$.
\end{lem}
\begin{proof}
The map $T(F)_{\sr} \times G(F) \to G(F)$ sending $(t,g) \mapsto gtg^{-1}$
is a smooth map of locally $F$-analytic manifolds.
By \cite[\S II.III.10]{serre-lie-algebras}, the image of the map is open.
Since the centralizer in $G(F)$ of any element of $T(F)_{\sr}$ is $T(F)$,
it follows that $[T(F)_{\sr}//T(F)]$ can be identified with an open
substack of $[G(F)_{\sr}//G(F)]$.  The complement of the image can be covered
with inertia stacks of other tori, which are open.  So the image is also closed.
\end{proof}

Define a function
\[ \mathcal{T}^{G_b \to G}_{b,\mu} \colon \Dist(G_b(F)_{\sr},\Lambda)^{G_b(F)} \to \Dist(G(F)_{\sr},\Lambda)^{G(F)} \]
as follows.  Given $\phi \in \HDist(G(F)_{\sr},\Lambda)^{G(F)}$
and a maximal torus $T$ of $G$,
denote the restriction of $\phi$ to $\HDist(T(F)_{\sr},\Lambda)$ by $\phi|_T$.
For any $\phi \in \Dist(G_b(F)_{\sr},\Lambda)^{G_b(F)}$, define $\mathcal{T}^{G_b \to G}_{b,\mu}(\phi)$
such that for any $g \in G(F)_{\sr}$,
\begin{equation} \label{transfer}
\left. \mathcal{T}^{G_b \to G}_{b,\mu}(\phi)\right|_{Z_G(g)} = (-1)^{\left<\mu,2\rho_G\right>} \sum_{(g,g',\lambda) \in \Rel_b} \dim r_{\mu}[\lambda] \left. \phi\right|_{Z_{G_b}(g')} \,.
\end{equation}
Here, $Z_G(g)$ and $Z_{G_b}(g')$ denote the centralizers of $g$ and $g'$, respectively,
$\rho_G$ is half the sum of the positive roots of $G$,
$r_{\mu}$ is the representation of the Langlands dual group $\widehat{G}$ with highest weight $\mu$,
and we use the fact that $g$ and $g'$ are stably conjugate to identify $Z_G(g)$ with $Z_{G_b}(g')$.

The following extends \cite[Proposition 6.4.7]{hkw}.
\begin{prop} \label{trace torsion}
Let $\Lambda$ be a $\Zl$-algebra that is killed by a power of $\ell$, and let $\rho$ be an admissible representation of $G_b(F)$
with coefficients in $\lambda$.  Then we have an equality
\[ \trdist\left(R\Gamma(G,b,\mu)[\rho]_{\sr} \right) = \sum_{b' \in \overline{\{b\}}} \mathcal{T}^{G_{b'} \to G}_{b',\mu} \cc_{\Bun_{G,C}/C}(i_{b*} \rho)_{b',\sr} \]
in $\Dist(G(F)_{\sr},\Lambda)^{G(F)}$.
\end{prop}
\begin{proof}
We will abbreviate $\Hecke_C$ as $H$, $\Bun_{G,C}$ as $B$, and $\In_C$ as $\In$.

For $b' \in \overline{\{b\}}$,
we consider the following diagram, which extends \cite[(6.4.2)]{hkw}.
\[ \begin{tikzcd}
\In(H^{b',1})_{\sr} \arrow[rr,"\widetilde{i}_{b'}"] \arrow[dd,"\widetilde{i}'_1"] & &
\In(H^{*,1})_{\sr} \arrow[r,"\In(h_2^{*,1})_{\sr}"] \arrow[d,"j_1'"] &
\In(B^1)_{\sr} \arrow[d,"j_1"] \\
& &
\In(H^{*,1}) \arrow[r,"\In(h_2^{*,1})"] \arrow[d,"\In(i_1')"] &
\In(B^1) \arrow[d,"\In(i_1)"] \\
\In(H^{b',*})_{\sr} \arrow[r,"j_{b'}'"] \arrow[d,"\In(h_1^{b',*})_{\sr}"] &
\In(H^{b',*}) \arrow[r,"\In(i_{b'}')"] \arrow[d,"\In(h_1^{b',*})"] &
\In(H) \arrow[r,"\In(h_2)"] \arrow[d,"\In(h_1)"] & \In(B) \\
\In(B^{b'})_{\sr} \arrow[r,"j_{b'}"]  & \In(B^{b'}) \arrow[r,"\In(i_{b'})"] & \In(B)
\end{tikzcd} \]

As in the proof of \cite[Proposition 6.4.7]{hkw}, we have
\[j_1^* \cc_{B^1/C}(i_* T_V^{\vee}(i_b)_* \rho) \cong (\In(h_2^{*,1})_{\sr})_* (j_1')^* \In(i_1')^* \cc_{H/C} (h_1^*(i_b)_* \rho \otimes \mathcal{S}_{V^{\vee}}) \,. \]
By Theorem \ref{sr open},
\[ \In(H)_{\sr} = \bigsqcup_{b' \in B(G)} \In(H^{b',*})_{\sr} \,. \]
The characteristic class of $(i_b)_* \rho$ is only supported on $\overline{\{b\}}$.
Therefore,
\begin{align*}
& (j_1')^* \In(i_1')^* \cc_{H/C} (h_1^*(i_b)_* \rho \otimes \mathcal{S}_{V^{\vee}}) \\
= & \sum_{b' \in \overline{\{b\}}}(j_1')^* \In(i_1')^* \In(i_{b'}')_* (j_{b'}')_* (j_{b'}')^* \In(i_{b'})^* \cc_{H/C} (h_1^*(i_b)_* \rho \otimes \mathcal{S}_{V^{\vee}}) \\
= & \sum_{b' \in \overline{\{b\}}}(\widetilde{i}_{b'}')_* (\widetilde{i}_1')^* (j_{b'}')^* \In(i_{b'}')^* \cc_{H/C} (h_1^*(i_b)_* \rho \otimes \mathcal{S}_{V^{\vee}}) \,,
\end{align*}
where we used proper base change in the last line.  

Again following the argument of loc.~cit.,
we obtain
\begin{align*} & (\widetilde{i}_{b'}')_*(\widetilde{i}_1')^* (j_{b'}')^* \In(i_{b'}')^* \cc_{H/C} (h_1^*(i_b)_* \rho \otimes \mathcal{S}_{V^{\vee}})  \\
= & (\widetilde{i}_{b'}')_*(\widetilde{i}_1')^* \left(j_{b'}^* \In(i_{b'})^* \cc_{B/C} ((i_b)_* \rho) \boxtimes_{\In(B^{\loc}_m)_{\sr}} \cc_{H^{\loc}_m/C}(\mathcal{S}_{V^{\vee}})_{\sr} \right) \,. \end{align*}

Then we have a commutative
diagram
\[ \begin{tikzcd}
H^0(\In(B^{b'})_{\sr},K_{\In(B^{b'})/C}) \arrow[r,"\sim"]
\arrow[d,"-\boxtimes_{\In(B^{\loc}_m)_{\sr}} \cc_{H^{\loc}_m/C}(\mathcal{S}_{V^{\vee}})_{\sr}"]
&
\Dist(G_{b'}(F)_{\sr},\Lambda)^{G_{b'}(F)} \arrow[d,"q_1^*(-)\otimes (-1)^{\left<2\rho_G,\mu\right>} \rank V^{\vee}{[}-{]}"] \\
H^0(\In(H^{b',*})_{\sr},K_{\In(H^{b',*})/C}) \arrow[r,"\sim"] \arrow[d,"(\widetilde{i}_1')^*"] &
\Dist(\pi_0(\Fix(\alpha_{b'})_{\sr}),\Lambda)^{G_{b'}(F)} \arrow[d,"(p_1)^*"] \\
H^0(\In(H^{b',1})_{\sr},K_{\In(H^{b',1})/C}) \arrow[r,"\sim"] \arrow[d,"(\widetilde{i}_{b'}')_*"] &
\Dist(\pi_0(\Fix(\alpha_{\Sht})_{\sr}),\Lambda)^{G(F) \times G_{b'}(F)} \arrow[d,"(p_2)_*"] \\
H^0(\In(H^{*,1})_{\sr},K_{\In(H^{*,1})/C}) \arrow[r,"\sim"] \arrow[d,"\In(h_2^{*,1})"] &
\Dist(\Fix(\alpha_1)_{\sr},\Lambda)^{G(F)} \arrow[d,"(q_2)_*"] \\
H^0(\In(B^1)_{\sr},K_{\In(B^1)/C}) \arrow[r,"\sim"] &
\Dist(G(F)_{\sr},\Lambda)^{G(F)}
\end{tikzcd} \,. \]

With these modifications, the argument of \cite{hkw} also works in
our more general situation.
\end{proof}

Finally, we arrive at our main result.

\begin{thm} \label{trace}
Let $\rho$ be an admissible finite length representation of $G_b(F)$, and let $\pi = R\Gamma(G,b,\mu)[\rho])$.
Let $g \in G(F)$.  Suppose that for all specializations $b'$ of $b$ in $B(G,\mu)$, $g$ is not stably conjugate to any
element of $G_{b'}(F)$.  Then
\[ \Theta_{\pi}(g) = (-1)^{\left<\mu,2\rho_G\right>} \sum_{(g,g',\lambda) \in \Rel_b} \dim r_{\mu}[\lambda] \Theta_{\rho}(g') \prod_{\alpha \in \Phi^+} |1-\alpha(\lambda)|^{-1} \,. \]
Here,
$\rho_G$ is half the sum of the positive roots of $G$,
$\Phi^+$ is the set of roots of $G$ that are positive with respect to the parabolic $P_b^+$,
and the norm $\abs{\cdot}$ is defined so that $\abs{p}=p^{-[F:\Qp]}$.

\end{thm}

\begin{proof}
This follows from Proposition \ref{trace torsion} in
the same way that \cite[Theorem 6.5.2]{hkw} follows from \cite[Proposition 6.4.7]{hkw}.
There is a factor of
\[ \prod_{\alpha \in \Phi^-} \abs{1-\alpha(\lambda)} \] coming from
Proposition \ref{cc formula simple} and a factor of
\[ \prod_{\alpha \in \Phi^+ \cup \Phi^-} \abs{1-\alpha(\lambda)}^{-1} \]
coming from the Weyl integration formula.  Here,
$\Phi^-$ is the set of roots of $G$ that are negative with respect to $P_b^+$.
\end{proof}

In some cases (such as in Examples \ref{gl2 example}--\ref{depth one example}),
we have more information about the characteristic class
of $i_{b*} \rho$, and we can compute $\Theta_{\pi}$ for
a larger set of $g$.
\begin{ex} \label{gl2 example hecke}
Let $G = \GL_2$, let $b$ correspond to the vector bundle
$\OO(1/2)$ on the Fargues--Fontaine curve, and let
$\mu$ be minuscule.
Let $\pi = R\Gamma(\Sht_{G,b,\mu})[\rho]$,
where $\rho$ is the trivial representation of $G_b(F)$.
Then it is known that
the image of
$\pi$ in $\Groth(\GL_2(F),\Qlb)$
is $\mathrm{St}-1$, where $\mathrm{St}$ is the Steinberg representation
and $1$ is the trivial representation.

Hansen \cite{hansen-blog} suggested that it would be
interesting to see how the above formula for $\pi$ arises
from the characteristic class framework.
We will give a derivation, using
Example \ref{gl2 example} and
Proposition \ref{trace torsion}.
It suffices to show that $\pi$ and $\mathrm{St}-1$
have the same Harish-Chandra character.

The Harish-Chandra character of the trivial representation is given by
\[ \Theta_{1}(g) = 1 \,. \]
The representation $\mathrm{St}\oplus 1$ is induced from the trivial
representation of a Borel subgroup of $\GL_2(F)$.
By van Dijk's formula \cite[Theorem 3]{vandijk},
\[ \Theta_{\mathrm{St}}(g) + \Theta_1(g) = \begin{cases}
\abs{1-\lambda_1/\lambda_2}^{-1} + \abs{1-\lambda_2/\lambda_1}^{-1} & \text{eigenvalues }\lambda_1,\lambda_2\text{ of }g\text{ are in }F^{\times} \\
0 & \text{otherwise}
\end{cases} \,.
\]
(Here, we have defined the norm on $F$ by $\abs{p}=p^{-[F:\Qp]}$.)
So
\begin{equation} \label{drinfeld hc}
\Theta_{\mathrm{St}}(g) - \Theta_1(g) = \begin{cases}
-2 + \abs{1-\lambda_1/\lambda_2}^{-1} + \abs{1-\lambda_2/\lambda_1}^{-1} & \text{eigenvalues }\lambda_1,\lambda_2\text{ of }g\text{ are in }F^{\times} \\
-2 & \text{otherwise}
\end{cases} \,.
\end{equation}
We need to show that $\Theta_{\pi}(g)$ equals the right-hand side of
\eqref{drinfeld hc}.

The set $B(G,\mu)$ contains two elements.  One of these is $b$.
Denote the other by $b'$; it corresponds to the vector bundle
$\OO \oplus \OO(1)$.

Let $g$ be a strongly regular element of $G(F)$.
If the eigenvalues of $g$ are not in $F^{\times}$, then $g$ is stably conjugate
to an element of $G_{b}(F)$, but not to an element of $G_{b'}(F)$.
If $g'$ is stably conjugate to $g$, then $\Theta_{\rho}(g') = 1$.
Then Theorem \ref{trace} implies that $\Theta_{\pi}(g)=-2$, in agreement
with \eqref{drinfeld hc}.

Now suppose $g$ has eigenvalues $\lambda_1,\lambda_2 \in F^{\times}$.
Then $g$ is stably
conjugate to the elements
$(\lambda_1,\lambda_2)$ and $(\lambda_2,\lambda_1)$ of $G_{b'}(F)$,
but not to any element of $G_b(F)$.  By Example \ref{gl2 example},
the Harish-Chandra character of
the $b'$-part of the characteristic class of $\rho$ takes the values
\[ \abs{1-\lambda_1/\lambda_2}\abs{|1-\lambda_2/\lambda_1} - \abs{1-\lambda_1/\lambda_2} \,, \]
\[ \abs{1-\lambda_1/\lambda_2}\abs{1-\lambda_2/\lambda_1} - \abs{1-\lambda_2/\lambda_1} \]
at $(\lambda_1,\lambda_2)$ and $(\lambda_2,\lambda_1)$, respectively.
Now apply Proposition \ref{trace torsion}.
After multiplying by the factor of $(-1)^{\left<\mu,2\rho_G\right>}=-1$
appearing in \eqref{transfer},
and dividing by a factor of
\[ \abs{1-\lambda_1/\lambda_2}\abs{1-\lambda_2/\lambda_1} \]
coming from the Weyl integration formula, we obtain
\[ \Theta_{\pi}(g) = -2+\abs{1-\lambda_1/\lambda_2}^{-1}+\abs{1-\lambda_2/\lambda_1}^{-1} \,,  \]
in agreement with \eqref{drinfeld hc}.
\end{ex}

\bibliography{bundles}{}

\begin{thebibliography}{GHW22}

\bibitem[FF18]{fargues-fontaine}
L.~Fargues and J.-M. Fontaine.
\newblock {\em Courbes et fibr{\'e}s vectoriels en th{\'e}orie de {Hodge}
  {{\(p\)} }-adique}, volume 406 of {\em Ast{\'e}risque}.
\newblock Paris: Soci{\'e}t{\'e} Math{\'e}matique de France (SMF), 2018.

\bibitem[FS21]{fargues-scholze}
L.~Fargues and P.~Scholze.
\newblock Geometrization of the local {Langlands} correspondence, 2021.
\newblock arXiv:2102.13459.

\bibitem[GHW22]{ghw}
D.~Gulotta, D.~Hansen, and J.~Weinstein.
\newblock An enhanced six-functor formalism for diamonds and v-stacks, 2022.
\newblock {arXiv}:2202.12467.

\bibitem[Gro67]{EGA44}
A.~Grothendieck.
\newblock \'{E}l\'{e}ments de g\'{e}om\'{e}trie alg\'{e}brique. {IV}. \'{E}tude
  locale des sch\'{e}mas et des morphismes de sch\'{e}mas {IV}.
\newblock {\em Inst. Hautes \'{E}tudes Sci. Publ. Math.}, (32):361, 1967.

\bibitem[Han21a]{hansen-blog}
D.~Hansen.
\newblock {H.-Kaletha-Weinstein} study guide and {FAQ}, 2021.
\newblock
  \url{https://totallydisconnected.wordpress.com/2021/05/06/h-kaletha-weinstein-study-guide-and-faq/}.

\bibitem[Han21b]{hansen-moduli}
D.~Hansen.
\newblock Moduli of local shtukas and {Harris}'s conjecture.
\newblock {\em Tunis. J. Math.}, 3(4):749--799, 2021.

\bibitem[HC99]{hc-admissible}
Harish-Chandra.
\newblock {\em Admissible invariant distributions on reductive {{\(p\)}}-adic
  groups. {Notes} by {Stephen} {DeBacker} and {Paul} {J}. {Sally} jun},
  volume~16 of {\em Univ. Lect. Ser.}
\newblock Providence, RI: American Mathematical Society, 1999.

\bibitem[HK25]{hansen-kedlaya-sheafiness}
D.~Hansen and K.~S. Kedlaya.
\newblock Sheafiness criteria for {Huber} rings, 2025.
\newblock \url{https://kskedlaya.org/papers/criteria.pdf}.

\bibitem[HKW22]{hkw}
D.~Hansen, T.~Kaletha, and J.~Weinstein.
\newblock On the {Kottwitz} conjecture for local shtuka spaces.
\newblock {\em Forum Math. Pi}, 10:79, 2022.
\newblock Id/No e13.

\bibitem[Hub96]{huber-etale}
R.~Huber.
\newblock {\em {\'E}tale cohomology of rigid analytic varieties and adic
  spaces}, volume E30 of {\em Aspects Math.}
\newblock Wiesbaden: Vieweg, 1996.

\bibitem[KL15]{kl-relative}
K.~S. Kedlaya and R.~Liu.
\newblock Relative {$p$}-adic {H}odge theory: foundations.
\newblock {\em Ast\'{e}risque}, (371):239, 2015.

\bibitem[Kot85]{kottwitz-isocrystals}
R.~E. Kottwitz.
\newblock Isocrystals with additional structure.
\newblock {\em Compos. Math.}, 56:201--220, 1985.

\bibitem[Sch17]{diamonds}
P.~Scholze.
\newblock {\'Etale} cohomology of diamonds, 2017.
\newblock arXiv:1709.07343.

\bibitem[Ser92]{serre-lie-algebras}
J.-P. Serre.
\newblock {\em Lie algebras and {Lie} groups. 1964 lectures, given at {Harvard}
  { University}.}, volume 1500 of {\em Lect. Notes Math.}
\newblock Berlin etc.: Springer-Verlag, 2nd ed. edition, 1992.

\bibitem[SGA03]{SGA1}
{\em S{\'e}minaire de g{\'e}om{\'e}trie alg{\'e}brique du {Bois} {Marie}
  1960-61. {Rev{\^e}tements} {\'e}tales et groupe fondamental ({SGA} 1). {Un}
  s{\'e}minaire dirig{\'e} par {Alexander} {Grothendieck}. {Augment{\'e}} de
  deux expos{\'e}s de {M}. {Raynaud}.}, volume~3 of {\em Doc. Math. (SMF)}.
\newblock Paris: Soci{\'e}t{\'e} Math{\'e}matique de France, {\'e}dition
  recompos{\'e}e et annot{\'e}e du original publi{\'e} en 1971 par {Springer}
  edition, 2003.

\bibitem[Ste65]{steinberg-regular}
R.~Steinberg.
\newblock Regular elements of semisimple algebraic groups.
\newblock {\em Inst. Hautes \'{E}tudes Sci. Publ. Math.}, (25):49--80, 1965.

\bibitem[SW20]{berkeley}
P.~Scholze and J.~Weinstein.
\newblock {\em Berkeley lectures on {$p$}-adic geometry}, volume 207 of {\em
  Annals of Mathematics Studies}.
\newblock Princeton University Press, Princeton, NJ, 2020.

\bibitem[vD72]{vandijk}
G.~van Dijk.
\newblock Computation of certain induced characters of {P}-adic groups.
\newblock {\em Math. Ann.}, 199:229--240, 1972.

\end{thebibliography}
\bibliographystyle{shortalpha}
\end{document}